\documentclass[reqno, 12pt]{amsart}

\usepackage[utf8]{inputenc}

\usepackage{tikz-cd}

\usepackage{amsfonts, amsthm, amsmath, amssymb}
\usepackage{hyperref}
\hypersetup{colorlinks=false}
\usepackage[all]{xy}

\usepackage{pifont}

\usepackage[margin=1in]{geometry}

\RequirePackage{mathrsfs} \let\mathcal\mathscr
  
\usepackage[colorinlistoftodos,french,bordercolor=black,backgroundcolor=red,linecolor=red,textsize=footnotesize,textwidth=0.75in,shadow]{todonotes}

\usepackage{hyperref}
\usepackage{dsfont}
\usepackage{upgreek}
\usepackage{mathabx,yfonts}
\usepackage{enumerate}

\numberwithin{equation}{section}

\newtheorem{theorem}{Theorem}[section] 
\newtheorem{lemma}[theorem]{Lemma}
\newtheorem{proposition}[theorem]{Proposition}
\newtheorem{corollary}[theorem]{Corollary}
\newtheorem{conjecture}[theorem]{Conjecture}

\theoremstyle{definition}
\newtheorem*{acknowledgements}{Acknowledgements}
\newtheorem{remark}[theorem]{Remark}

\newtheorem*{notation}{Notation}

\newcommand{\e}{\mathrm{e}}

\renewcommand{\phi}{\varphi}

\renewcommand{\leq}{\leqslant}

\renewcommand{\geq}{\geqslant}

\renewcommand{\c}{\mathbf{c}}

\renewcommand{\b}{\mathbf{b}}

\renewcommand{\r}{\mathbf{r}}

\newcommand{\da}{\mathrm{d}\alpha}

\DeclareSymbolFont{bbold}{U}{bbold}{m}{n}
\DeclareSymbolFontAlphabet{\mathbbold}{bbold}

\newcommand{\md}[1]{  \left(\textnormal{mod}\ #1\right)}
 
\renewcommand{\P}{\mathbb{P}}
\newcommand{\Q}{\mathbb{Q}}

\newcommand{\F}{\mathbb{F}}
\newcommand{\N}{\mathbb{N}}
\newcommand{\R}{\mathbb{R}}

\newcommand{\Z}{\mathbb{Z}}
\renewcommand{\l}{\left}

\renewcommand{\r}{\right}
\renewcommand{\b}{\mathbf}
\renewcommand{\c}{\mathcal}
\renewcommand{\epsilon}{\varepsilon}

\renewcommand{\leq}{\leqslant}
\renewcommand{\geq}{\geqslant}

\title
[Rational points and prime values of polynomials] 
{Rational points and prime values of polynomials \\ in moderately many variables}

\author{Kevin Destagnol}
\address{
IST Austria\\
Am Campus 1\\
3400 Klosterneuburg\\
Austria}
\email{kevin.destagnol@ist.ac.at}

\author{Efthymios Sofos} 
\address{
Max Planck Institute for Mathematics\\
Vivatsgasse 7, Bonn, 53111, Germany}
\email{sofos@mpim-bonn.mpg.de}

\date{\today}

\begin{document}

\begin{abstract}
We   
derive the Hasse principle and weak approximation 
for 
fibrations 
of certain varieties
in the spirit of work by  
Colliot-Thélène--Sansuc
and 
Harpaz--Skorobogatov--Wittenberg.
Our varieties are defined through polynomials in many variables
and part of our work is devoted to establishing 
Schinzel's hypothesis
for polynomials of this kind. This last part is achieved by using arguments behind Birch's well-known result regarding the Hasse principle for complete intersections with the notable difference that we prove our result in 50\% fewer variables than in the classical Birch setting. We also study the problem of square-free values of an integer polynomial with 66.6\% fewer variables than in the Birch setting. 
\end{abstract}

\maketitle

\setcounter{tocdepth}{1}
\tableofcontents

\section{Introduction}
\label{introduction}

\subsection{Prime values of polynomials and rational points}
\label{s:apparithm}
 Let $n\geqslant 1$ be an integer and assume that $f \in \mathbb{Q}[t_1,\dots,t_n]$.
	Let $K_1,\ldots,K_r$ be cyclic extensions of $\Q$, denote
	the degree $[K_i:\Q]$ by $d_i$ and fix a basis $\{\omega_{1,i},\dots,\omega_{d_i,i}\}$ for $K_i$ as a vector space over $\mathbb{Q}$. We will denote
$$
	\mathbf{N}_{K_i/\mathbb{Q}}(\mathbf{x}_i)= N_{K_i/\mathbb{Q}}(x_{1,i}\omega_{1,i}+\cdots+x_{d_i,i}\omega_{d_i,i}), \quad (1\leqslant i \leqslant r)
$$
	where $N_{K_i/\mathbb{Q}}$ denotes the field norm. Let now the quasi-affine variety  $X\subset \mathbb{A}^n\times \mathbb{A}^{d_1}\times \cdots \times \mathbb{A}^{d_r}$  be defined via 
\begin{equation}
	X:
	\big(0 \neq f(t_1,\dots,t_n)=\mathbf{N}_{K_1/\mathbb{Q}}(\mathbf{x}_1)=\cdots=\mathbf{N}_{K_r/\mathbb{Q}}(\mathbf{x}_r) \big)
		\label{def:nonaff}
	\end{equation}
	and  let $V$ be a smooth proper model of the affine subvariety of $\mathbb{A}^n\times \mathbb{A}^{d_1}\times \cdots \times \mathbb{A}^{d_r}$ given by   
	\begin{equation}
	 f(t_1,\dots,t_n)=\mathbf{N}_{K_1/\mathbb{Q}}(\mathbf{x}_1)=\cdots=\mathbf{N}_{K_r/\mathbb{Q}}(\mathbf{x}_r) 
	.
		\label{def:nonaff2}
	\end{equation}
The Hasse principle and weak approximation for 
varieties of this kind 
have been the object of intensive study.
There are cases where the Hasse principle and weak approximation hold and 
there are examples for which they fail, \cite{Colliot}. 
However, it has been conjectured by Colliot-Thélène  \cite{Colliot2} that all such failures are accounted for by the Brauer--Manin obstruction.

The main objective of this paper is to study the Hasse principle and weak approximation for the class of varieties defined by (\ref{def:nonaff}) and (\ref{def:nonaff2}) under the restriction that the polynomial $f$ is an irreducible form and has many variables compared to its degree, but only moderately so as it will appear in due course. To this end, information on
prime values 
assumed by integer
polynomials can be exploited.
The prototypical example is due to Hasse~\cite{Ha},
whose proof of the Hasse principle for smooth
quadratic forms in four variables relies on Dirichlet's theorem on primes in arithmetic progressions combined 
with the global reciprocity law and the Hasse principle for non-singular
quadratic forms in three variables. 
This fibration argument was later generalized 
in an important work
by Colliot-Thélène and Sansuc~\cite{Colliot} to establish that, 
conditionally under Schinzel's hypothesis, various 
pencils of varieties over $\mathbb{Q}$ satisfy the Hasse principle and weak approximation. 
Their result was then extended by many authors,
see the introduction of~\cite{WSko} for a list of relevant references. 
%
Theorem~\ref{thm:main} below will allow 
us
to replace Schinzel's hypothesis in order to prove unconditionally the Hasse principle and weak approximation for the varieties defined by (\ref{def:nonaff}) and (\ref{def:nonaff2}) in the case of an irreducible form $f$ with moderately many 
variables compared to its degree.
\noindent
Let us conclude by mentioning that unconditional proofs in this subject
exist in  cases where the underlying polynomials 
have small degree (see for instance \cite[Th. 9.3]{CTSSw}) or special factorisation over $\Q$. For example, polynomials that are 
completely split over $\mathbb{Q}$ are treated in~\cite{BLS}.
Our main result provides an example where the polynomial
needs to have moderately many variables compared to its degree
but has no restriction on its shape. 

For a homogeneous
polynomial $f\in \Z[x_1,\ldots,x_n]$ 
that is irreducible in
  $\Q[x_1,\ldots,x_n]$
we let $\sigma_f$ be the dimension of the singular locus of $f=0$, namely the dimension of the affine variety cut out by the system of equations $\nabla f(x_1,\dots,x_n)=0$ (see \cite[pg. 250]{birch}).
Observe that $0\leq \sigma_f \leq n-1$ and that $\sigma_f=0$
if and only if the projective variety defined by $f$ is non-singular.

\begin{theorem}
	\label{hpp}
	Let $f,K_i$ and $X$ be as in (\ref{def:nonaff}) with $f$ an irreducible form
	and assume that 
	$$n-\sigma_f\geq \max\{4,1+2^{\deg(f)-1}(\deg(f)-1)\}.$$
Then $X$ satisfies the Hasse principle and weak approximation. In particular, $X(\Q)$ is Zariski dense as soon as it is non-empty.
\end{theorem}
Note that, thanks to the fact that our Theorem~\ref{thm:main} below (which is the main ingredient of the proof of Theorem~\ref{hpp})
holds in half as many variables as in the work of Birch,
a direct application of~\cite[\S 7,Th. 1]{birch} would not prove Theorem~\ref{hpp}.
Our strategy will be to establish an analogue of \cite[Prop. 1.2]{WSko} and then to adapt the argument in the
proof of \cite[Th. 1.3]{WSko}. 

\indent
Like
in the proof of~\cite[Th. 9.3]{CTSSw},
one can deduce weak approximation for the variety $V$ defined by (\ref{def:nonaff2}) from Theorem~\ref{hpp}, since weak approximation is a birational invariant of smooth varieties and hence it is enough to establish the result for the smooth model of $V$ provided by $X$. We then conclude the proof of Corollary~\ref{corr} by alluding to the fact that the Hasse principle and Zariski density by the existence of a rational point are consequences of weak approximation.
\begin{corollary}
	Keep the assumptions of Theorem~\ref{hpp}
	and
	let $V$ be 
	as
	in (\ref{def:nonaff2}). 
Then $V$ satisfies the Hasse principle and weak approximation. In particular, $V(\Q)$ is Zariski dense as soon as it is non-empty.
\label{corr}
\end{corollary}

\subsection{Primes represented by polynomials
}
\label{s:gensha}
As mentioned above, a key tool in our proof of Theorem~\ref{hpp} is a generalization of Schinzel's hypothesis for polynomials in moderately 
many variables compared to its degree. 
Let 
$f\in \Z[x_1,\ldots,x_n]$ be a, not necessarily homogeneous, polynomial that is irreducible in 
$\Q[x_1,\ldots,x_n]$
and
denote by 
$f_0$
the top degree
part of~$f$.
We define $\sigma_f:=\sigma_{f_0}$. For a 
non-empty
compact box
$\c{B}\subset \R^n$ 
with the property that
$f_0(\c{B})\subset  
(1,\infty), 
$
we
define 
\begin{equation}
\label{def:lif}
\pi_f(\c{B})
:=\#\{\b{x} \in \Z^n \cap \c{B}:f(\b{x}) \text{ is a positive
prime}\}
\
\
\text{ and } 
\
\
\mathrm{Li}_f(\c{B})
:=\int_{\c{B}}\frac{\mathrm{d}\b{x}}{\log f_0(\b{x})}
.\end{equation}
Our main result in this section is the following theorem.

\begin{theorem}
\label{thm:main}
Assume that $f\in \Z[x_1,\ldots,x_n]$ 
is any integer 
polynomial
which is irreducible in
 $\Q[x_1,\ldots,x_n]$
and let 
$\c{B}\subset \R^n$ be any non-empty compact
box with  $f_0(\c{B})\subset 
(0,\infty)$. 
If 
\begin{equation}
\label{halfassumption}
n-\sigma_f 
\geq \max\{4,(\deg(f)-1) 2^{\deg(f)-1}+1\}
,\end{equation}
then for every fixed $A>0$ the following holds for all sufficiently large 
$P$,
\[
\pi_f(P\c{B})=
\left(
\prod_{p\ \text {prime }}
\frac{
(1-p^{-n}\#\{
\b{x}\in \F_p^n:f(\b{x})=0\})
}
{(1-1/p)}
\right)
\mathrm{Li}_f(P\c{B}) 
+O_{A,\c{B},f}
\l(\frac{P^n}{ (\log P)^{A} }   \r)
 \] where the implied constant  depends at most on $A,\c{B}$ and $f$.
\end{theorem}
Note that the assumption 
$f_0(\c{B})\subset 
(0,\infty)$ shows that for all sufficiently large $P$
every $\b x\in P\c B$ satisfies $f_0(\b x) \gg P^{\deg(f)}>1$, thus $\mathrm{Li}_f(P\c{B}) $
is well-defined.
We shall see in Lemma~\ref{lem:laurent} that $
\mathrm{Li}_f(P\c{B})
-
\mathrm{vol}(\c{B})
P^n/
\log(P^{\deg(f)}
) \ll_{f,\c{B}}
P^n/(\log P)^2
$, hence \begin{equation}\label
{eq:collusion}
\hspace{-0,2cm}
\pi_f(P\c{B})
\!=\!
\frac{\mathrm{vol}(\c{B})}{\deg(f)}
\left(
\prod_{p\ \text {prime }}
\!
\frac{
(1-p^{-n}\#\{
\b{x}\in \F_p^n:f(\b{x})=0\})
}
{(1-1/p)}
\right)
\frac{P^n}{\log P}
+O_{f,\c{B}}
\l(\frac{P^n}{ (\log P)^{2} }   \r).
\end{equation}
\indent
Bateman and Horn~\cite{BatH} provided heuristics that led to a conjecture regarding 
the prime values of 
an integer polynomial in a single variable.
One can modify their heuristics in the case that the polynomial has arbitrarily many variables,
thus resulting in an analogous conjecture regarding the prime values of an integer polynomial in many variables. We refer the reader to Appendix~\ref{s:bathorco} for a quick overview of Schinzel's hypothesis, the Bateman-Horn conjecture and their generalisations. Theorem~\ref{thm:main} then establishes the analogous conjecture provided that the polynomial has sufficiently many variables compared to its degree.

There are currently no
available techniques capable of
settling any case of the Bateman--Horn conjeture
in one variable
apart from the case of one
linear polynomial, which is Dirichlet's theorem for primes in arithmetic progressions.
Efforts have therefore focused on settling such problems for polynomials in  more variables.
Notable examples in
cases with $n=2$ 
are 
Iwaniec's work~\cite{I} for quadratic polynomials, 
Fouvry--Iwaniec's work~\cite{FII} for $x_1^2+x_2^2$ with $x_2$ prime,
Friedlander--Iwaniec's work~\cite{FI} for $x_1^2+x_2^4$ ,
Heath-Brown's work~\cite{HB} for $x_1^3+2x_2^3$,
Heath-Brown--Moroz's work~\cite{HBM} for binary cubic forms
and the recent work of
Heath-Brown--Li~\cite{HBL} on $x_1^2+x_2^4$ with $x_2$ prime.
The special shape of these polynomials 
plays a central r\^ole in the proofs of these results;
they are all related to norms of a number field.
In cases with $n>2$ 
it should be noted that 
Green--Tao--Ziegler~\cite{GTZ}
studied
simultaneous prime values of certain linear polynomials
by a variety of methods,
Friedlander and Iwaniec~\cite{MR2486486} 
studied the prime values of $x_1^2+x_2^2+x_3^2$ 
via the class number formula of Gauss,
while
Maynard's work~\cite{Maynard}
employs 
geometry of numbers
to
cover the case of
incomplete norm forms.

It is therefore 
a natural question 
whether the problem 
of representing 
primes by polynomials
can be studied 
for polynomials with no special shape. Let us recall here that one of the important
theorems in the frontiers between analytic number theory and 
Diophantine geometry 
concerns the Hasse principle 
for
systems of polynomials in many variables and with no special shape
by Birch~\cite{birch}.
To prove 
Theorem~\ref{thm:main} 
we shall employ the Hardy--Littlewood
circle method in the form used by Birch
and use several of his estimates. 

While Birch's work applies to  every 
non-singular
homogeneous polynomial $f$ having at least 
$
n
\geq
(\deg(f)-1) 2^{\deg(f)}+1
$
variables
(which was
recently improved 
by Browning and Prendiville~\cite{TimSean}
to 
$
n
\geq
(\deg(f)-\sqrt{\deg(f)}/2) 2^{\deg(f)}
$
),
the assumption~\eqref{halfassumption} of our
Theorem~\ref{thm:main}
is  less
restrictive,
as it allows for half as many variables.  
The
improved range
is due to
the use of 
$L_2$-norm
inequalities
in the minor arcs,
as well as 
bounds for exponential sums
due to Browning--Heath-Brown~\cite{41}
and 
Deligne~\cite{deligne}
to show that the singular series 
in Birch's work 
converges absolutely in the range~\eqref{halfassumption}. 

Let us finally give a direct consequence of 
Theorem~\ref{thm:main}.
\begin{corollary}
\label{cor:kebab}
Let  $f\in \Z[x_1,\ldots,x_n]$ 
be an
integer 
homogeneous
polynomial
which is irreducible in
 $\Q[x_1,\ldots,x_n]$
and assume that  $ n-\sigma_f 
\geq \max\{4,(\deg(f)-1) 2^{\deg(f)-1}+1\}.
$ 
Then $f(\b x)$
takes infinitely many distinct 
positive
prime
values 
as $\b x$ ranges over $ \Z^n$ 
if and only if 
$ 
f(\R^n)$ is not included in
$(-\infty,0] $
 and  for every prime $p$
the set 
$f(\Z^n)$ is not included in
$p\Z$.
\end{corollary} \begin{proof}
We clearly need to focus only on the sufficiency.
If $f(\R^n)\nsubset (-\infty,0]$
holds, then we can obviously find a non-empty box $\c{B}\subset \R^n$
with $f(\c{B})\subset  (0,+\infty) $, so that 
$\mathrm{vol}(\c B)\neq 0$. 
If $f(\Z^n)\nsubset  p\Z$
holds, then the $p$-adic factor in~\eqref{eq:collusion}
is strictly positive and we shall see in Lemma~\ref{lem:padicdens}
that the product over $p$
is absolutely convergent. 
Hence by~\eqref{eq:collusion} we deduce that 
$\pi_f(P)\asymp_{f,\c B} P^n/\log P$. If $f(\b x)=q$ was soluble only for finitely many primes $q$,
say $q_1,\ldots,q_r$,
then the standard
estimate 
$\#\{\b x\in \Z^n\cap P\c B:f(\b x)=q\}\ll_{f,q,\c B} P^{n-1}$
would lead to 
\[\frac{P^n}{\log P}\asymp_{f, \c B}
\pi_f(B)=\sum_{i=1}^r
\#\{\b x\in \Z^n\cap P\c B:f(\b x)=q\}
\ll_{f,q_1,\ldots,q_r, \c B}
P^{n-1}
,\] which is a contradiction. \end{proof}

\subsection{Square-free integers represented by polynomials}
\label{s:noabc}
An integer $m$ is called square-free if for every prime $p$ we have 
$p^2\nmid m$. In particular, $0$ is not square-free and  $m$ is square-free if and only if $-m$ is.

Assume that we are given a polynomial
$f\in \Z[x_1,\ldots,x_n]$
 that is separable as an element of $\Q[x_1,\ldots,x_n]$
and let $f_0$ and $\sigma_f$ be as in~\S\ref{s:gensha}.
A similar approach to the one for Theorem~\ref{thm:main} allows us to study 
the set 
$
\b{S}_f:=\{\b{x}\in \Z^n: f(\b{x}) \text{ is square-free}\}.$
\begin{theorem}
\label{thm:second}
Assume that $f\in \Z[x_1,\ldots,x_n]$ 
is any integer 
polynomial
which is
separable as an element of $\Q[x_1,\ldots,x_n]$
and let 
$\c{B}\subset \R^n$ be any non-empty closed
box. If 
\begin{equation}
\label
{23assumption}
n-\sigma_f 
> \max\Big\{1,\frac{1}{3}(\deg(f)-1) 2^{\deg(f)}\Big\}
,\end{equation}
then there exists $\beta=\beta(f)>0$ such that for all $P\geq 2$ the equality 
\[
\frac{\#\{\b{S}_f\cap P\c{B}\}}{\#\{\Z^n\cap P\c{B}\}}
=
\prod_{p\ \text {prime }}
\Big(1-p^{-2n}\#\big\{\b{x}\in (\Z/p^2\Z)^n: f(\b{x})=0 
\big\}\Big)
+O_{f,\c{B}}(P^{-\beta})
\]
holds with an implied constant that depends at most on $f$ and $\c{B}$.
\end{theorem}
The problem of square-free values of integer polynomials
has a very
long history, see~\cite{timsquarefree}
for a list of references.
Many cases are still open, for example, 
there is no irreducible 
quartic integer
polynomial in one variable for which we know that it takes infinitely many square-free 
values.
One of the most general results, 
conditional on the $abc$ conjecture,
is due to Poonen~\cite{MR1980998}
where arbitrary polynomials
are treated.
Our Theorem~\ref{thm:second} 
covers
unconditionally
arbitrary  polynomials of fixed degree and number of variables
with the proviso that the number of variables is suitably large compared to the degree.
Theorem~\ref{thm:second} features a saving of
two thirds of the 
variables
compared to the Birch setting~\cite{birch}.
This saving comes from the fact that 
exponential sums whose terms 
are
restricted to square-free integers 
can be bounded in a satisfactory manner, this was done in the work of 
Br\"udern, Granville, Perelli, Vaughan and Wooley~\cite{MR1620828}
and 
Keil~\cite{MR3043604}.

\begin{notation}
We shall
use the
notation $\mathbf{x}$ to refer
to $n$-tuples $\b x = (x_1,...,x_n)$. 
We will also make use of the classical von Mangoldt function denoted $\Lambda$ and of the classical M\"{o}bius function denoted $\mu$. 
The letter $d$ will refer exclusively to 
the degree of the polynomial~$f$ in Theorem~\ref{thm:main}.
Finally,
throughout
the paper, we shall 
make
use of
the notation 
\begin{equation}
\label{exp}
\e(z):=\exp(2\pi i z),
z\in \mathbb{C}.
\end{equation}
The polynomial $f$ and the box $\c B$ 
will be considered 
fixed 
throughout. This is taken to mean that, although 
each implied constant in the big $O$ notation will depend on several quantities related to $f$,
we shall avoid recording these dependencies. The list of the said quantities consists of
\[f_0,n,d,\sigma_f, \theta_0, \delta, \eta, \lambda_1,A,\lambda,
\]
whose meaning will become evident in due course.
The symbol $\epsilon$ will be used for a small positive parameter whose value may vary, allowing, for example, inequalities of the form 
$x^\epsilon\ll x^{\epsilon/4}$. 
Further dependency of the implied constants on other quantities 
will be recorded explicitly via an appropriate use of subscript.

\end{notation}
\begin{acknowledgements}
We are grateful to Jean-Louis
Colliot-Thélène
and Yang Cao for helpful 
conversations regarding the applications of Theorem~\ref{thm:main}.
We would also like to thank Tim Browning for 
suggesting the proof of  
Proposition~\ref{deligne}.
We wish to acknowledge the comments of the anonymous referee 
that helped improve this work considerably.
\end{acknowledgements}

\section{The proof of Theorem~ \ref{hpp}}
\label{s:prfinal}
Denote by $\mathbb{Q}_v$ the completion of $\mathbb{Q}$ with respect to the place $v$,
let $|\cdot|_p$ be the $p$-adic norm defined by $|x|_p=p^{-\nu_p(x)}$ for $x \in \mathbb{Q}_p$ if $v=p$ is finite and define
$|\cdot|_{\infty}$ as the 
classical absolute value for the real place. 
We will use
the notation
$\mathbb{Z}_S$ for the ring of $S$-integers
for any finite set of finite places $S$ and we will say that a prime $p$ is a fixed prime divisor of a polynomial $f \in \mathbb{Z}[x_1,\dots,x_n]$ if, for all $(x_1,\dots,x_n) \in \mathbb{Z}^n$, we have $p \mid f(x_1,\dots,x_n)$.

\subsection{Preliminary lemmas}
We begin by establishing the following analogue of \cite[Lem. 2]{Colliot}.
\begin{lemma}
Let $f \in \mathbb{Z}[x_1,\dots,x_n]$ be a non-zero polynomial  
with content equal to $1$. If $p$ is such that $f\left(\mathbb{Z}^n\right) \subseteq p\mathbb{Z}$, then $p \leqslant \deg(f)$.
	\label{fpf}
\end{lemma}
\begin{proof}
Define $d:=\deg(f)$ and 
let $p$ be a prime such that $f\left(\mathbb{Z}^n\right) \subseteq p\mathbb{Z}$. On one hand, we have by assumption that
$$
\#\left\{ \mathbf{x} \in \left(\mathbb{Z}/p\mathbb{Z}\right)^n \hspace{1mm} : \hspace{1mm} f(\mathbf{x})= 0 
 \right\}=
p^n.
$$
On the other hand, since $f$  has content one, \cite[Eq.(2.7)]{SW} implies that
$$
\#\left\{ \mathbf{x} \in \left(\mathbb{Z}/p\mathbb{Z}\right)^n \hspace{1mm} : \hspace{1mm} f(\mathbf{x})=0
  \right\}\leqslant d p^{n-1}
$$
and hence $p\leqslant d$, thus concluding the proof of the lemma.
\end{proof}

We now
use
Lemma~\ref{fpf} to verify
the following analogue of~\cite[Prop. 1.2]{WSko} and of the hypothesis ($H_1$) of \cite{CS94} over $\mathbb{Q}$.

\begin{proposition}
 	Let $f \in \mathbb{Q}[x_1,\dots,x_n]$ be an irreducible homogeneous polynomial satisfying the assumptions (\ref{halfassumption}) and $f(1,0,\dots,0)>0$. Let $C$ be a positive real constant and $\varepsilon>0$. Suppose we are given $(\lambda_{1,p},\dots,\lambda_{n,p}) \in \mathbb{Q}_p^n$ for $p$ in a finite set of finite places $S$ containing all primes $p \leqslant \deg(f)$  
	and all primes $p$ such that $f$ does
	not have $p$-integral coefficients as well as all primes $p$ such that $\nu_p(f(1,0,\dots,0))>0$. Then there exists infinitely many $(\lambda_1,\dots,\lambda_n) \in \mathbb{Z}_S^n$ such that $\lambda_1>C\lambda_i>0$ for all $i \in \{2,\dots,n\}$,
	$
	|\lambda_i-\lambda_{i,p}|_p<\varepsilon
	$
	for all $i \in\{1,\dots,n\}$ and $p \in S$ and $f(\lambda_1,\dots,\lambda_n)=\ell u$ for a prime $\ell \notin S$ and $u\in \mathbb{Z}_S^{\times}$, $u>0$.
	\label{analogue}
\end{proposition}
\begin{proof}
	Up to multiplication of $(\lambda_1,\dots,\lambda_n)$ and $(\lambda_{1,p},\dots,\lambda_{n,p})$ by a product of powers of primes in $S$, we can assume without loss of generality that $(\lambda_{1,p},\dots,\lambda_{n,p}) \in \mathbb{Z}_p^n$ for $p \in S$. The assumption that $f(1,0,\dots,0)>0$ provides
	with $a_{\mathbf{i}} \in \mathbb{Q}$ and $a>0$ such that 
	\begin{equation}
	\label{a}
	f(x_1,\dots,x_n)=ax_1^d+\sum_{\substack{\mathbf{i}=(i_1,\dots,i_n) \in \mathbb{N}^n\\i_1+\cdots+i_n=d \\ 0 \leqslant i_1,\dots,i_n \leqslant d \\ i_1 \neq d}} a_{\mathbf{i}}x_1^{i_1}\cdots x_n^{i_n}.
	\end{equation}
	 Let  $N_{<0}$
	  be the number of indices $\b{i}$ with  $a_{\mathbf{i}} <0$.
	We can  assume that 	$C>\frac{N_{>0} |a_{\mathbf{i}}|  }{a}$
and
	$C>1$ 
 	 whenever $a_{\mathbf{i}}<0$. 
	As in the proof of \cite[Prop. 1.2]{WSko}, 
	we can now
	find
	$(\lambda_{0,1},\dots,\lambda_{0,n}) \in \mathbb{Z}^n$ such that $|\lambda_{0,i}-\lambda_{i,p}|_p<\varepsilon/2$ for all $p \in S$. We can choose them such that $\lambda_{0,1}>C\lambda_{0,i}>0$ for all $i \in \{2,\dots,n\}$. 
	We can now see that   $f(\lambda_{0,1},\dots,\lambda_{0,n})>0$
	by alluding to 
	$$
	f(\lambda_{0,1},\dots,\lambda_{0,n})\geqslant a\lambda_{0,1}^d-\sum_{a_{\mathbf{i}}<0}|a_{\mathbf{i}}|\lambda_{0,1}^{i_1}\cdots \lambda_{0,n}^{i_n},
	$$
	and the inequalities
	$$
	a\lambda_{0,1}^d=\frac{a}{N_{>0}}\sum_{a_{\mathbf{i}}<0}\lambda_{0,1}^d=\frac{a}{N_{>0}}\sum_{a_{\mathbf{i}}<0}\lambda_{0,1}^{i_1}\cdots\lambda_{0,1}^{i_n}>\sum_{a_{\mathbf{i}}<0}\frac{a}{N_{>0}}C^{d-i_1}\lambda_{0,1}^{i_1}\cdots\lambda_{0,n}^{i_n}>\sum_{a_{\mathbf{i}}<0}|a_{\mathbf{i}}|\lambda_{0,1}^{i_1}\cdots\lambda_{0,n}^{i_n}
	.$$
	Let $A=\prod_{p\in S}p$
	and a fixed integer $N$ big enough so that $|A^N|_p <\varepsilon/2$ for all $p \in S$.
	Now consider the polynomial $g \in \mathbb{Q}[x_1,\dots,x_n]$ given by 
	$$
	g(x_1,\dots,x_n)=f(\lambda_{0,1}+x_1A^N,\dots,\lambda_{0,n}+x_nA^N).
	$$
	The polynomial $g$ can be expressed as $g=t\tilde{g}$ for $t \in \mathbb{Z}_S^{\times}$ and $\tilde{g}$ a polynomial with integer coefficients which is irreducible over $\mathbb{Q}$. Let us denote by $c$ the product of all fixed prime factors of $\tilde{g}$. We will now establish that if $p$ is a prime factor of $c$ then $p \in S$. Let $p \mid c$. Either $p$ divides the content of $\tilde{g}$ and in particular, with the notation (\ref{a}), $\nu_p(aA)\neq 0$ which immediately implies that $p \in S$ or, denoting by $\tilde{c}$ the content of $\tilde{g}$, $p$ is a fixed prime factor of the polynomial $\tilde{g}/\tilde{c}$ which has integral coefficients and content equal to one.
By Lemma \ref{fpf} this implies that $p \leqslant \deg(f)$ and hence that $p \in S$. Moreover, with the notation of \S 1.2, $\tilde{g}_0=A^{dN}f$
	and the conditions $x_1>Cx_i>0$ define an open cone $\mathcal{C}$ in $\mathbb{R}^n$. In addition, 
	when $f$ is evaluated at $(\lambda_{0,1},\dots,\lambda_{0,n})\in \c{C}$ it produces a strictly
	positive value,
	therefore 
	we can find a box $\mathcal{B} \subseteq \mathcal{C}$ such that $f(\c{B})\subset (0,\infty)$.
	Since for all $P$ we have $P\mathcal{B} \subset \mathcal{C}$, we obtain from 
	Theorem~\ref{thm:main} that
	there exist infinitely many $\mathbf{x}\in \Z^n \cap \mathcal{C}$ 
	such that $\tilde{g}(\mathbf{x})/c$ is prime.	 Introducing $
	\lambda_i=\lambda_{0,i}+x_iA^N$ for any $i \in \{1,\dots,n\}$ and for any such $\mathbf{x} \in \Z^n \cap \mathcal{C}$, we get that
	$$
	f(\lambda_1,\dots,\lambda_n)/(ct)=g(x_1,\dots,x_n)/(ct)
	$$
	is prime. This yields the result because $\lambda_{0,1}>C\lambda_{0,i}>0$ and $x_{1}>C x_{i}>0$, which implies 
	$\lambda_1=(\lambda_{0,1}+x_1A^N)>C\lambda_i=C(\lambda_{0,i}+x_iA^N)$. Moreover, $|\lambda_i-\lambda_{i,0}|_p\leqslant |A^N|_p<\varepsilon/2$ and hence $|\lambda_i-\lambda_{i,p}|_p<\varepsilon$ for all $p \in S$ and all $i \in \{1,\dots,n\}$.
	\qedhere
\end{proof}
\subsection{Conclusion of the proof of 
Theorem~\ref{hpp}} 
\begin{proof}
We proceed by adapting the proof of \cite[Th. 1.3]{WSko}. We are given
$1>\varepsilon>0$,
a finite set of places $S$  
and 
 a point $(\mathbf{t}_v,\mathbf{x}_{1,v},\dots,\mathbf{x}_{r,v}) \in X(\mathbb{Q}_v)$ for every place $v$ of $\mathbb{Q}$ and we want to find $(\mathbf{t},\mathbf{x}_{1},\dots,\mathbf{x}_{r}) \in X(\mathbb{Q})$ such that for all $v \in S$, 
 $$
 \left\{\begin{aligned}
 &\left|t_i-t_{i,p}\right|_v<\varepsilon \qquad (i \in \{1,\dots,n\}),\\
 &\left|x_{j_i,i}-x_{j_i,i,v}\right|_v<\varepsilon \qquad (i \in \{1,\dots, r\}, \hspace{0.5cm} j_i \in \{1,\dots,d_i\}).
 \end{aligned}
 \right.
 $$
 \subsubsection{First step}
By density and continuity, we can assume that $\mathbf{t}_{\infty} \in \mathbb{Q}^n$ and by a linear change of variables, we can assume that $\mathbf{t}_{\infty}=(1,0,\dots,0)$. Note that the solubility over~$\mathbb{R}$ implies
that $f(1,0,\dots,0)>0$ in the case where there is a totally imaginary $K_i$.
In addition, it implies that 
$f(1,0,\dots,0)$ can be
strictly positive or strictly negative
when all $K_i$ are totally real. 
We denote by $s\in \{-1,+1\}$ the sign of $f(1,0,\dots,0)$. 
We can enlarge $S$ so that it contains the real place, the field $K_i$ is unramified outside $S$ for all $i \in \{1,\dots,r\}$, $S$ contains all primes $p \leqslant \deg(f)$, all primes $p$ such that $f$ does not have $p$-integral coefficients as well as primes $p$ such that $\nu_p(f(1,0,\dots,0))>0$. 
 \subsubsection{Second step}
Let $L=[d_1,\dots,d_r]$ denote the least common multiple of the degrees $d_1,\dots,d_r$
  and $M=\max_{v\in S}\max_{1\leqslant i \leqslant r \atop 1 \leqslant j_i \leqslant d_i}|x_{j_i,i,v}|_v$. By \cite[Prop. 6.1]{FV}, we know that the image of the map $N_{K_{i,p}/\mathbb{Q}_p}:K_{i,p}^{\times} \rightarrow \mathbb{Q}_p^{\times}$ is 
	open and that a polynomial function is continuous for the $p$-adic topology. Hence there exists $\varepsilon'>0$ such that for every $(\lambda_1,\dots,\lambda_n) \in \mathbb{Q}^n$ satisfying for every $i \in \{1,\dots,n\}$, the inequality $
	|\lambda_i-t_{i,p}|_p<\varepsilon'
	$, we have that $f(\lambda_1,\dots,\lambda_n)$ is a
 	local norm for $K_i/\mathbb{Q}$ for the place $p$ and there exists $\varepsilon''>0$ such that for every $(\lambda_1,\dots,\lambda_n) \in \mathbb{Q}^n$ satisfying for every $i \in \{1,\dots,n\}$, the inequality $
	|\lambda_i-t_{i,p}|_p<\varepsilon''
	$, we have that
	\begin{equation}
	|f(\lambda_1,\dots,\lambda_n)-f(t_{1,p},\dots,t_{n,p})|_p<\frac{\varepsilon}{2M}|f(t_{1,p},\dots,t_{n,p})|_p|L|_p, \qquad (p \in S).
	\label{auxxx}
	\end{equation}
 Then
applying
Proposition~\ref{analogue} yields $(\lambda_1,\dots,\lambda_n) \in \mathbb{Z}_S^n$ such that $\lambda_1>C\lambda_i>0$ for all $i \in \{2,\dots,n\}$,
	$
	|\lambda_i-t_{i,p}|_p<\min\{\varepsilon/2,\varepsilon',\varepsilon''/2\}
	$
	for all $i \in\{1,\dots,n\}$ and $p \in S$ and $f(\lambda_1,\dots,\lambda_n)=s\ell u$ for a prime $\ell \notin S$ and $u\in \mathbb{Z}_S^{\times}$
	with $u>0$. We thus obtain that $f(\lambda_1,\dots,\lambda_n)$ is a
	local norm for $K_i/\mathbb{Q}$ for all places of $S$. This is also the case for the real place because $f(\lambda_1,\dots,\lambda_n)>0$ in the case that
	there is a 
	totally imaginary
	$K_i$. 
	\subsubsection{Third step}
	Now, $f(\lambda_1,\dots,\lambda_n)=\ell su$ is a unit in $\mathbb{Z}_p$ for every $\mathbb{Q}_p$ and $p \notin S \cup \{\ell\}$ and we know by 
	\cite[Prop.
	V.3.11]{Janusz} that this implies that $f(\lambda_1,\dots,\lambda_n)$ is a local norm for $K_i/\mathbb{Q}$ for all $p \notin S \cup \{\ell\}$. By the global reciprocity law and the fact that $K_i/\mathbb{Q}$ is unramified outside $S$ we see
	that $f(\lambda_1,\dots,\lambda_n)$ is also a local norm for $K_i/\mathbb{Q}$ at the place $\ell$ (see \cite[Prop. V.12.9]{Janusz}). 
	The conclusion
	is that 
	$f(\lambda_1,\dots,\lambda_n)$ is a local norm for $K_i/\mathbb{Q}$ at every place of $\mathbb{Q}$ and then by the Hasse norm 
	principle 
	\cite[Th. V.4.5]{Janusz}, 
	one gets that there exists $(\mathbf{x}_1,\dots,\mathbf{x}_r)\in \mathbb{Q}^{d_1}\times \cdots \times \mathbb{Q}^{d_r}$ such that
	$
	0 \neq f(\lambda_1,\dots,\lambda_n)=\mathbf{N}_{K_1/\mathbb{Q}}(\mathbf{x}_1)=\cdots=\mathbf{N}_{K_r/\mathbb{Q}}(\mathbf{x}_r).
	$
	\subsubsection{Fourth step}
	By continuity, there exists $\varepsilon_1>0$ such that for all $q_1\in \Q^{\times}$ such that $|q_1-\lambda_1|_{\infty}<\varepsilon_1$, then $\left|\frac{1}{q_1}-\frac{1}{\lambda_1} \right|_{\infty}<\frac{\varepsilon}{2\max_{1\leqslant i \leqslant n} |\lambda_i|_{\infty}}$. Writing $m=[L,\deg(f)]$, by weak approximation in~$\mathbb{Q}$, since $\lambda_1>0$ one can find $\rho \in \mathbb{Q}$ such that $|\rho-1|_p<\min\left\{\frac{\varepsilon}{2},\frac{\varepsilon''}{2}\right\}\min\{1,|\lambda_i|_p\}$ for all $p \in S$ and $|\rho^{m/\deg(f)}-\lambda_1|_{\infty}<\min\left\{\frac{\varepsilon}{2}\min\left\{1,\lambda_1-\frac{\varepsilon}{2}\right\},\varepsilon_1\right\}$ where we can assume that $\varepsilon<2\lambda_1$. In particular, this implies that $|\rho|_p=1$ for all $p \in S
	$. We now
	make the following change of variables,
$$
\lambda_i=\rho^{m/\deg(f)} \lambda'_i, \quad i \in \{1,\dots,n\}, \quad \mathbf{x}_i=\rho^{m/d_i}\mathbf{x}'_i, \quad i \in \{1,\dots,r\}
,$$
so that for all finite place $p \in S$
we have 
$$
|\lambda_i'-\lambda_i|_p=|\lambda'_i|_p|\rho^{m/\deg(f)}-1|_p\leqslant |\lambda_i|_p|\rho-1|_p<\frac{\varepsilon}{2}
$$
and therefore $|\lambda'_i-t_{i,p}|_p<\varepsilon$ for all $i \in \{1,\dots,n\}$. Moreover, we have
\begin{equation}
0\neq f(\lambda'_1,\dots,\lambda'_n)
=
\mathbf{N}_{K_1/\mathbb{Q}}(\mathbf{x}'_1)=\cdots=\mathbf{N}_{K_r/\mathbb{Q}}(\mathbf{x}'_r).
\label{a2}
\end{equation}
As for the real place, we have $|\lambda'_1-1|_{\infty}<\varepsilon$ and 
$| \lambda'_i-\lambda_i/\lambda_1|_{\infty}<\varepsilon/2$. 
The treatment of the archimedian place is now
concluded similarly as in~\cite{WSko},
by alluding to $0<\lambda_i/\lambda_1<C^{-1}$ and by 
taking $C$ big enough, namely $C>\frac{2}{\varepsilon}$.
  \subsubsection{Fifth step} To conclude the proof, it remains to find $(\mathbf{x}''_1,\dots,\mathbf{x}''_r)\in \mathbb{Q}^{d_1}\times \cdots \times \mathbb{Q}^{d_r}$ $v$-adically close to $(\mathbf{x}_{1,v},\dots,\mathbf{x}_{r,v})$ for all $v \in S$ and such that $\mathbf{N}_{K_i/\mathbb{Q}}(\mathbf{x}''_i)=\mathbf{N}_{K_i/\mathbb{Q}}(\mathbf{x}'_i)$ for all $i \in \{1,\dots,r\}$. By (\ref{auxxx}) and the choice of 
 $\boldsymbol{\lambda}'=(\lambda'_1,\dots,\lambda'_n)$ above, one can write $f(\lambda'_1,\dots,\lambda'_n)=f(t_{1,p},\dots,t_{n,p})\beta_p$ with $\beta_p \in \Q_p$ satisfying $|\beta_p-1|_p<\frac{\varepsilon}{2M}|L|_p$ for all finite place $p \in S$. In particular, $|\beta_p|_p=1$ and $\beta_p \in \Z_p^{\times}$ for all finite place $p \in S$. Now, Hensel's lemma implies that there exists $\alpha_p \in \Z_p$ such that $\beta_p=\alpha_p^L$ and 
$$
|\alpha_p-1|_p=\left|\frac{\beta_p-1}{L}\right|_p<\frac{\varepsilon}{2M}.
$$
Of course, there exists $\alpha_{\infty}$ such that $\beta_{\infty}=\alpha_{\infty}^L$ and $|\alpha_{\infty}-1|_{\infty}<\varepsilon/(2M)$ since one can always ensure that $f(1,0,\dots,0)$ and $f(\lambda'_1,\dots,\lambda'_n)$ have the same sign.
Now alluding to the facts that $(\mathbf{t}_v,\mathbf{x}_{1,v},\dots,\mathbf{x}_{r,v}) \in X(\mathbb{Q}_v)$ and to (\ref{a2}), we obtain that for all $v \in S$
$$
0\neq f(\lambda'_1,\dots,\lambda'_n)
=
\mathbf{N}_{K_1/\mathbb{Q}}(\alpha_v^{L/d_1}\mathbf{x}_{1,v})=\cdots=\mathbf{N}_{K_r/\mathbb{Q}}(\alpha_v^{L/d_r}\mathbf{x}_{r,v}).
$$
In other words, for every $i \in \{1,\dots,r\}$, we have $\mathbf{N}_{K_i/\mathbb{Q}}\left(\alpha_v^{L/d_i}\mathbf{x}_{i,v}\right)=\mathbf{N}_{K_i/\mathbb{Q}}\left(\mathbf{x}'_i\right)$ for all $v \in S$. Thanks to the fact that weak approximation holds for the norm tori $N_{K_i/\mathbb{Q}}(\mathbf{z})=1$, one gets the existence of $\mathbf{x}''_i \in \Q^{d_i}$ such that $|x''_{j,i}-\alpha_v^{L/d_i}x_{j,i,v}|_v<\varepsilon/2$ for all $v \in S$ and $j \in \{1,\dots,d_i\}$ and $\mathbf{N}_{K_i/\mathbb{Q}}(\mathbf{x}''_{i})=\mathbf{N}_{K_i/\mathbb{Q}}(\mathbf{x}'_{i})$. Therefore, we have
$$
0\neq f(\lambda'_1,\dots,\lambda'_n)
=
\mathbf{N}_{K_1/\mathbb{Q}}(\mathbf{x}''_{1})=\cdots=\mathbf{N}_{K_r/\mathbb{Q}}(\mathbf{x}''_{r})
$$
along with
$$
|x''_{j,i}-x_{j,i,v}|_v\leqslant |x''_{j,i}-\alpha_v^{L/d_i}x_{j,i,v}|_v+|x_{j,i,v}|_v|\alpha_v-1|_v<\varepsilon,
$$
thus concluding the proof of Theorem~\ref{hpp}.
\end{proof}

\section{The proof of Theorem~\ref{thm:main}} \label{s:conclus}

\subsection{First steps and auxiliary estimates}   \label{s:utilising}
The
proof of 
Theorem~\ref{thm:main}
is initiated by 
using the following exponential sums for
real $\alpha$,
\begin{equation}
\label{s}
{S(\alpha):=
\sum_{\mathbf{x} \in \mathbb{Z}^n \cap  P\mathcal{B} }
\e\left( \alpha f(\mathbf{x})\right)}
\
\text{ and }
\
{W(\alpha):=
\sum_{\frac{1}{2}
\min\{f_0(\c{B})\} P^d \leq
p \leq 
2\max\{f_0(\c{B})\} P^d}
\e(\alpha p)
,}
\end{equation}
where 
we used that, 
in the setting of 
Theorem~\ref{thm:main},
the 
following
quantities are positive
\[
\min\{f_0(\c{B})\}
=
\min\{f_0(\b{x}):\b{x}\in \c{B}\}
\ \text{ and }  \
\max\{f_0(\c{B})\}
:=\max\{f_0(\b{x}):\b{x}\in \c{B}\}
.\]
The fact that
$\int_0^1 \e(\alpha \{f(\b{x})-p\})\mathrm{d}\alpha$
is $1$ when $f(\b{x})=p$ and is otherwise 
$0$, shows that for all $P\gg_{f,\mathcal{B}} 1$, we have the equality
\begin{equation}
\label{eq:clear}
\pi_f(P\c{B})
=
\int_0^1
S(\alpha)
\overline{W(\alpha)}
\mathrm{d}\alpha
.\end{equation}
This identity has the useful feature that it completely
separates the problem of evaluating 
$\pi_f$ into two problems, one regarding the evaluation of
the sum $S$ 
(that is only related to the values of the polynomial $f$)
and one regarding the evaluation of the sum $W$ 
(that is only related to the distribution of primes).
Birch~\cite{birch} has a similar identity, save for the factor $\overline{W(\alpha)}$.
The main idea is that the presence of this extra factor 
can be turned to our advantage, 
as it attains small
values for certain $\alpha$ for which $|S(\alpha)|$
is large. 
Let us comment that we could have defined $W$
in an alternative way by replacing the range for the primes $p$
by the condition 
$\min\{f_0(\c{B})\} P^d \leq p \leq \max\{f_0(\c{B})\} P^d$, however, our choice will 
make 
more transparent
the proof of Lemma~\ref{lem:prwinoskafes}. 

Before proceeding let us recall here the estimates from the work of Birch~\cite{birch}
that we shall need later.
First, 
following~\cite[pg. 251,Eq.(5)]{birch},
we let 
for  $\theta \in (0,1]$
and
$a \in \Z\cap[0,q)$
with $\gcd(a,q)=1$,
\begin{equation}
\mathcal{M}_{a,q}(\theta)
:=
\Big\{\alpha\in (0,1]:2|q\alpha-a |\leq P^{-d+(d-1)\theta}
\Big\}
\label{maq}
\end{equation}
and
\begin{equation}
\mathcal{M}(\theta)=\bigcup_{1\leq q \leq P^{(d-1)\theta}} \bigcup_{\substack{a \in \Z\cap[0,q)\\\gcd(a,q)=1}} \mathcal{M}_{a,q}(\theta).
\label{mtheta}
\end{equation}
\noindent
Birch then gives the following upper bound for the volume of $\mathcal{M}(\theta)$.
\begin{lemma}[Birch~{\cite[Lem. 4.2]{birch}}] 
\label{lem:birch4..22..}
$\c M(\theta)$ has volume at most 
$P^{-d+2(d-1)\theta}$.
\end{lemma}
Next, we choose any positive $\delta,\theta_0$ satisfying
\begin{equation}
\label{eq:10}
1>\delta+6 d \theta_0
\ \text{ and } \ 
\frac{n-\sigma_f}{2^{d-1}}-(d-1)>\delta \theta_0^{-1}
.\end{equation}
As in~\cite[pg.252,Eq.(13)-(14)]{birch}
it is easy to see that there exists $T\in \N$
 and positive 
real numbers $\theta_1,\ldots,\theta_T$ with the properties
\begin{equation}
\left\{\begin{aligned}	
&T\ll P^\delta,\\
&\theta_T>\theta_{T-1}>\ldots>\theta_1>\theta_0>0
,\\
& d=2(d-1)\theta_T\\
&\frac{1}{2}\delta>2(d-1)(\theta_{t+1}-\theta_t) \text{ for } 0\leq t \leq T-1
.
\end{aligned}
\right.
\label{thetai}
\end{equation}
We next recall~\cite[Lem. 4.3]{birch}. Note that it was proved for homogeneous $f$, but, as noted by
Schmidt~\cite[\S 9]{schmidt}
a similar argument  works for inhomogeneous $f$, because the Weyl differencing process is not affected by lower order terms.
\begin{lemma}[Birch~{\cite[Lem. 4.3]{birch}}] 
\label{lem:birch1233df}
Let $0<\theta\leq 1$ and $\epsilon>0$. Then if $\alpha$ is not in $\c M(\theta) \ { modulo } \ 1$,
\[|S(\alpha)|\ll
P^{n-
\theta \left(
\frac{n-
\sigma_f}{2^{d-1}  }\right)
+\epsilon}
.\]  
\end{lemma}
Following the notation
in~\cite[pg. 253]{birch} and for 
$\theta,a,q$ as above we also let 
\begin{equation}
\mathcal{M}'_{a,q}(\theta)
:=
\Big\{\alpha\in (0,1]: 
|q\alpha-a |\leq q P^{-d+(d-1)\theta}
\Big\}
\label{mpaq}
\end{equation}
and
\begin{equation}
\mathcal{M}'(\theta)=\bigcup_{1\leq q \leq P^{(d-1)\theta}} \bigcup_{\substack{a \in \Z\cap[0,q)\\\gcd(a,q)=1}} 
\mathcal{M}'_{a,q}(\theta).
\label{mptheta}
\end{equation}
With 
$\theta_0$ as in (\ref{eq:10}), we let 
\begin{equation}
\eta:=(d-1)\theta_0
\label{eta}
\end{equation}
and for $a\in \Z,q\in \N$ we define
\begin{equation}
S_{a,q}:=
\sum_{\b x \in (\Z/q\Z)^n}
\e\l(\frac{af(\b x)}{q}\r).
\label{saq}
\end{equation}
Finally, for any $\gamma \in \R$ and 
any
measurable $\c C\subset [-1,1]^n$ 
we define 
\begin{equation}
I(\c C;\gamma)
:=
\int_{\b x \in \c C}
\e\l(\gamma f_0(\b x)\r)
\mathrm{d}\b x.
\label{icg}
\end{equation}

The following result, due to Birch, gives an upper bound for the quantity $I(\c C;\gamma)$.

\begin{lemma}[Birch~{\cite[Lem. 5.2]{birch}}] 
\label
{lem:birch5.2}
Let $\c C$ be a box contained in $[-1,1]^n$ with sidelength at most $\sigma<1$.
Then \[I\l(\c C,\gamma \r)\ll \sigma^n 
\min \l[1,
(\sigma^d 
|\gamma |   
)
^{-\frac{n-\sigma_f}{2^{d-1}(d-1)}+\epsilon
}
\r]
.\] 
\end{lemma}

The next result was proved by Birch with $f$ instead of $f_0$ in the definition of $I(\c C;\gamma)$,
however the following result holds in light of the remarks concerning
the function  $\mu(\infty,\mathcal{B})$ appearing in Schmidt's work~\cite[\S 9]{schmidt}.
\begin{lemma}[Birch~{\cite[Lem. 5.1]{birch}}] 
\label{lem:birch5.1}
Assume that we are given coprime integers 
 $q\in \N$ and 
$a\in \Z\cap [0,q)$ and let 
$\alpha\in \c M'_{a,q}(\theta_0)$ with the notations (\ref{eq:10}) and (\ref{mpaq}).
Denoting  $\beta:=\alpha-\frac{a}{q}$,
we have 
\[
S(\alpha)=q^{-n} P^n
S_{a,q} I(\c B;P^d \beta)+O(P^{n-1+2\eta} )   .\]
\end{lemma}

Let us now turn to the quantity $S_{a,q}$ defined in (\ref{saq}). 
The next two lemmas will be used to prove 
Proposition~\ref{deligne}.

\begin{lemma}[Birch~{\cite[Lem. 5.4]{birch}}] 
\label{lem:birch5.4}
For every $\epsilon>0$ and for 
$a\in \Z$,
$q \in \N$  with $\gcd(a,q)=1$ we have 
\[S_{a,q}\ll
q^{n-\frac{n-\sigma_f}{2^{d-1}(d-1)}+\epsilon}
.\] 
\end{lemma}

\begin{lemma}
[Browning--Heath-Brown~{\cite[Lem. 25]{41}}
] 
\label
{lem25 RHB TB}
We have 
\[
S_{a,p^k} \ll_k 
p^{ ( k-1)n +\sigma_f}
\] for all $k\geq 2 $. 
\end{lemma}
Note that, as explained in~\cite[Eq. (6.1)]{41}, the quantity $\sigma$ in~\cite[Lem. 25]{41} coincides 
with $-1+\sigma_f$, with $\sigma_f$ as in the present work.

The next result is of key importance
in the proof of Theorem~\ref{thm:main}.  
It is what allows to save variables compared to the Birch setting.
 
\begin{proposition}
\label
{deligne}
Let $f\in \Z[x_1,\ldots,x_n]$
be an irreducible polynomial
and define 
\[
T_f(q):=q^{-n}\sum_{a\in (\Z/q\Z)^*}|S_{a,q}|
,
\
q\in \N.\]
$(1)$ If $n-\sigma_f \geq \max\{5,(\deg(f)-1) 2^{\deg(f)-1}+2\}$ then the 
abscissa
of convergence of the Dirichlet series of $T_f$ is 
strictly negative. \newline
$(2)$
If $n-\sigma_f \geq \max\{4,(\deg(f)-1) 2^{\deg(f)-1}+1\}$ then there exists
a constant  
$C'=C'(f)>0$
such that  
$\sum_{q\leq x} T_f(q) \ll (\log x)^{C'}$.
\end{proposition}
\begin{proof}
Part $(1)$.
It is sufficient to prove that there exists 
 $\lambda_1>0$ such that
$\sum_q q^{\lambda_1}T_f(q)<\infty$.
By \cite[\S 7]{birch},
the function $T_f$ is multiplicative, hence the series over $q$ converges absolutely
 if 
the analogous Euler product converges absolutely,
i.e.  
\begin{equation}
\label{eq:dequiv}
\sum_{p \text{ prime} \atop k\in \N} 
p^{k\lambda_1}T_f(p^k)
=
\sum_{p \text{ prime} \atop k\in \N} 
p^{k(\lambda_1-n)}
\sum_{a\in (\Z/p^k\Z)^*}|S_{a,p^k}|
<\infty
.\end{equation}
By Lemma~\ref{lem:birch5.4}
the
terms with $k>2^{d-1}(d-1)$ 
contribute
\begin{equation}
\label
{eq:congb}
\ll 
\sum_{p} \sum_{k\geq 1+2^{d-1}(d-1)}
p^{k\left(1-\frac{n-\sigma_f}{2^{d-1}(d-1)}+\epsilon+\lambda_1
\right)}
.\end{equation}
By the assumption
$n-\sigma_f\geq (d-1) 2^{d-1}+2$ 
\[
\frac{n-\sigma_f}{(d-1) 2^{d-1}}
\geq
1+
\frac{2}{(d-1) 2^{d-1}}
\]
we have  
\[p^{
\left(1-\frac{n-\sigma_f}{2^{d-1}(d-1)}+\epsilon+\lambda_1
\right)}
\leq p^{
\left(-\frac{2}{2^{d-1}(d-1)}+\epsilon+\lambda_1
\right)}
,\] thus taking $\epsilon,\lambda_1$ sufficiently small we can ensure that this is 
at most $p^{ -\frac{1}{2^{d-1}(d-1)} } 
\leq 
2^{ -\frac{1}{2^{d-1}(d-1)} } 
$, which is of the form $1-\delta$ for some $0<\delta<1$.
Note that if $\delta\in (0,1)$
then for all  
$z\in \R$ with 
$0\leq  z  \leq 1-\delta$
and all $k_0\in \N$ we have 
\[
\sum_{k \geq k_0}
z^k
= 
\frac{z^{k_0}  }{1-z}
\leq \frac{z^{k_0}
}{\delta}
.\]
Therefore the sum in~\eqref{eq:congb}
is  
\[\ll_d 
\sum_{p} p^{(1+2^{d-1}(d-1))\left(1-\frac{n-\sigma_f}{2^{d-1}(d-1)}+\epsilon+\lambda_1\right)}
\ll 
\sum_{p} p^{-2+(\epsilon+\lambda_1)(1+2^{d-1}(d-1))    }
\ll
\sum_{p} p^{-3/2}
<\infty
,\] where we have taken $\epsilon,\lambda_1$ sufficiently small to ensure 
$(\epsilon+\lambda_1)(1+2^{d-1}(d-1))  \leq 1/2$.
Next, we 
study the contribution
towards~\eqref{eq:dequiv}
of any $k\in [2,2^{d-1}(d-1)]$.
By \cite[Lem. 25]{41}
we infer that the said contribution is
\[
\ll \sum_p
p^{
k (\lambda_1-n)+k
+ (k-1)n+\sigma_f
}
=\sum_p p^{
(1+\lambda_1)k-n+
\sigma_f
} 
\leq 
\sum_p p^{ (1+\lambda_1)(2^{d-1}(d-1))-n+\sigma_f} 
.\]
The assumption
$n-\sigma_f \geq 
(d-1) 2^{d-1}+2$
shows
that 
the exponent is 
$
\leq
\lambda_1
(2^{d-1}(d-1))-2 
$
and
for small $\lambda_1$ 
the sum converges.
To conclude the proof of~\eqref{eq:dequiv} 
it only remains to 
bound the contribution of terms with $k=1$. As noted in \cite[\S 5]{TimSean}, one can prove
\begin{equation}
T_f(p)=p^{-n} \sum_{a \in \left(\Z/p\Z\right)^{*}}\left|S_{a,p}\right| \ll p^{1-\frac{n-\sigma_f}{2}}
\label{Tfp}
\end{equation}
by Deligne's estimate and induction on $\sigma_f$. Taking small
$\lambda_1<1/4$ and using the assumption $n-\sigma_f\geq 5$
shows that the terms with $k=1$ in (\ref{eq:dequiv}) form a convergent series.
This completes our proof.
\newline
Part $(2)$. 
If
$k\in [2,2^{d-1}(d-1)]$
then Lemma~\ref{lem25 RHB TB}
and
$n-\sigma_f\geq 1+2^{d-1}(d-1)$
imply that
$T_f(p^k)\ll p^{-1}$.
Furthermore,
using Lemma~\ref{lem:birch5.4}
and
$n-\sigma_f\geq 1+2^{d-1}(d-1)$
we have that if $k\geq 1+2^{d-1}(d-1)$
then 
$T_f(p^k)\ll p^{-1-2^{d}(d-1)^{-1}}$.
Finally,
$n-\sigma_f \geq 4$, thus 
(\ref{Tfp}) ensures that 
$T_f(p)\ll p^{-1}$.
Putting everything together yields
$
\sum_{k\geq 1}  T_f(p^k)
\leq C'p^{-1}
$ for some $C'=C'(f)>0$
and the proof is concluded by
using 
$
\sum_{q\leq x} T_f(q)
\leq \prod_{p\leq x} (1+\sum_{k\geq 1} T_f(p^k))
$.
\end{proof}

\subsection{The minor arcs}
\label{section:minor}
For $\theta \in (0,1]$
and
$a \in \Z\cap[0,q)$
with $\gcd(a,q)=1$
we use the sets 
$\mathcal{M}(\theta)$ and 
$\mathcal{M}_{a,q}(\theta)$
defined by~(\ref{mtheta}) and (\ref{maq}). Next, we choose any positive $\delta,\theta_0$ satisfying (\ref{eq:10}).

\begin{lemma}
\label{lem:4.4}
For any
$0<\theta\leqslant 1$
we have 
\[
\left|\int_{\alpha \notin \c{M}(\theta)} S(\alpha)\overline{W(\alpha)}\da\right|
\ll
\l(\int_{\alpha \notin \c{M}(\theta)} |S(\alpha)|^2\da\r)^{1/2}
P^{d/2} (\log P)^{-1/2}
.\]
\end{lemma}
\begin{proof}
By
Schwarz's inequality 
the integral on the left side 
is bounded by 
\[
\l(\int_{\alpha \notin \c{M}(\theta)} |S(\alpha)|^2\da\r)^{1/2}
\l(\int_0^1 |W(\alpha)|^2\da\r)^{1/2}
.\] 
The proof is concluded by noting that 
\[
\int_0^1 |W(\alpha)|^2\da
=
\sum_{
\frac{1}{2}
\min\{f_0(\c{B})\} P^d \leq
p \leq 2 \max\{f_0(\c{B})\} P^d}
1
\ll P^d/\log P. \qedhere\]
\end{proof}
\begin{lemma}
\label{lem:4.4b}
Keep the assumptions of
Theorem~\ref{thm:main}
and~\eqref{eq:10}.
Then we have,
\[
\l(\int_{\alpha \notin \c{M}(\theta_0)} |S(\alpha)|^2\da\r)^{1/2}
=O(P^{n-d/2-\delta/2})
.\]
\end{lemma}
\begin{proof}
Using
the entities $(\theta_i)_{i=0}^T$,
given in~(\ref{thetai}),
we have for sufficiently small $\epsilon>0$,
\[
\int_{\alpha \notin \c{M}(\theta_T)} |S(\alpha)|^2\da
\ll
P^{2\left(n-\left(\frac{n-\sigma_f}{2^{d-1}}\right)\theta_T\right)+\epsilon}\leq 
P^{2n-d-\delta}, 
\]
due to Lemma~\ref{lem:birch1233df}, the third equation of (\ref{thetai}) and~\eqref{halfassumption}. For $t<T$
and
$\epsilon>0$
we get 
\[
\int_{\c{M}(\theta_{t+1})\smallsetminus\c{M}(\theta_t)} |S(\alpha)|^2\da
\ll
P^{-d+2(d-1)\theta_{t+1}+2\left(n-\left(\frac{n-\sigma_f}{2^{d-1}}\right)\theta_t\right)+\varepsilon}
\]
by~Lemmas \ref{lem:birch4..22..} and \ref{lem:birch1233df}.
The proof can now
be completed easily
by using the last equation of (\ref{thetai}),
~\eqref{eq:10}
and $T\ll P^\delta$,
as in
the last stage 
of the proof of~\cite[Lem. 4.4]{birch}. 
\end{proof}
\begin{lemma}
\label{lem:4.4c}
Keep the assumptions of
Theorem~\ref{thm:main}
and~\eqref{eq:10}.
Then we have,
\[
\left|\int_{\alpha \notin \c{M}(\theta_0)} S(\alpha)\overline{W(\alpha)}\da\right|
=O(P^{n-\delta/2})
.\]
\end{lemma}
\begin{proof}
The proof follows immediately
by 
tying together
Lemmas~\ref{lem:4.4} and~\ref{lem:4.4b}.
\end{proof}
Recall the definition of 
$\mathcal{M}'(\theta_0)$ and $\mathcal{M}'_{a,q}(\theta_0)$
given in~(\ref{mptheta}) and (\ref{mpaq}). The next lemma is analogous to~\cite[Lem. 4.5]{birch}.
\begin{lemma}
\label{lem:4.4c}
Keep the assumptions of
Theorem~\ref{thm:main}
and~\eqref{eq:10}.
Then we have
\[
\pi_f(P\c{B})=
\sum_{q\leq P^{(d-1)\theta_0}}
\sum_{\substack{ a\in \Z\cap [0,q)\\ \gcd(a,q)=1}}
\int_{\c{M}'_{a,q}(\theta_0)} 
S(\alpha)\overline{W(\alpha)}
\da
+O(P^{n-\delta/2}) 
,\]
where $C'$ is as in Proposition~\ref{deligne}.
\end{lemma} 

Before proceeding we 
note that
one can take an arbitrarily small positive value for $\theta_0$ in Lemma~\ref{lem:4.4c}
because
the system of inequalities~\eqref{eq:10} can be solved for any $\theta_0>0$ small enough.
This will come at the cost of a worse error term in Lemma~\ref{lem:4.4c},
however,
it will still exhibit a power saving and it will thus be acceptable for the purpose
of verifying Theorem~\ref{thm:main}.
\subsection{The intermediate range}
\label{s:inter}
Under
the assumptions of
Theorem~\ref{thm:main}
and~\eqref{eq:10}
we can use Lemmas \ref{lem:birch4..22..}, \ref{lem:birch5.1}
and the trivial bound 
$W(\alpha)\ll P^d$
to evaluate 
the quantity $S(\alpha)$
in Lemma~\ref{lem:4.4c}. This yields
\begin{equation}
\label{eq:ents}
\frac{\pi_f(P\c{B})}{P^n}
-
\sum_{q\leq P^{(d-1)\theta_0}}
q^{-n}
\sum_{\substack{ a\in \Z\cap [0,q)\\ \gcd(a,q)=1}}
S_{a,q}
\int_{|\gamma|\leq P^{\eta}} 
I(\c{B};\gamma)
\frac{\overline{W(a/q+\gamma P^{-d})}}{P^{d}}
\mathrm{d}\gamma
\ll
(\log P)^{-A}
,\end{equation}
valid for all $A>0$,
where $\eta$,
$S_{a,q}$ and $I(\c{B};\gamma)$
are defined respectively
in~(\ref{eta}), (\ref{saq})
and~(\ref{icg}).

For $A,q\in \N$ and $a \in \Z\cap [0,q) $ with $\gcd(a,q)=1$ 
we let 
\begin{equation}
\label{eq:entsla}
\mathfrak{M}_{a,q}(A):=
\{\alpha \in \R \md{1}: |\alpha-a/q|\leq P^{-d} (\log P)^A\}
,\end{equation}
\begin{equation}
\label{eq:entslala}
\mathfrak{M}(A):=\bigcup_{1\leq q\leq (\log P)^{A}}
\bigcup_{\substack{a \in \Z\cap [0,q) \\ \gcd(a,q)=1}}
\mathfrak{M}_{a,q}(A)
\end{equation}
and we observe that 
$\mathfrak{M}(A) \subset \c{M}'(\theta_0)$
for all $P\gg 1$. 
We denote the difference by 
\begin{equation}
\label{eq:entslalala}
\mathfrak{t}(A)
:=
\c{M}'(\theta_0)
\setminus
\mathfrak{M}(A)
.\end{equation}
The set 
$\mathfrak{t}(A)$
is therefore
to be thought of as
lying `between' the major arcs 
$\c{M}'(\theta_0)$
and the minor arcs
$[0,1)\setminus \c{M}'(\theta_0)$.
We shall see in~\S\ref{s:major}
that 
$\mathfrak{M}(A)$
gives
rise to the main term 
in Theorem~\ref{thm:main}.

Next,
we observe
that Lemma~\ref{lem:birch5.2}
and our assumption $n-\sigma_f\geq 1+2^{d-1}(d-1)$
yield 
\begin{equation}
\label{eq:bound}
\int_{|\gamma|\geq Q} 
|I(\c{B};\gamma)|
\mathrm{d}\gamma
\ll 
Q^{-\frac{1}{2^{d}(d-1)}}
\ , \quad 
(Q\geq 1),
\end{equation}
in particular showing that 
$\int_\R |I(\c{B};\gamma)|
\mathrm{d}\gamma
$ converges under assumption~\eqref{halfassumption}. 
\begin{lemma}
\label{lem:vinogradov}
If~\eqref{halfassumption} holds
then  
\begin{equation}
\label
{tonny iommi > beethoven} 
\sum_{(\log P)^{A}<q\leq P^{\eta}}
q^{-n}
\sum_{\substack{ a\in \Z\cap [0,q)\\ \gcd(a,q)=1}}
|S_{a,q}|
\int_{|\gamma|\leq P^{\eta}} 
|I(\c{B};\gamma)|
\frac{|W(a/q+\gamma P^{-d})|}{P^{d}}
\mathrm{d}\gamma
\ll
(\log P)^{-A/2+3+C'}  
,\end{equation}
where $C'$ is as in Proposition~\ref{deligne}.
\end{lemma}
\begin{proof}
If
$\alpha$ is not in the union of the sets 
$\{\alpha \md{1}:|\alpha-a/q|\leq P^{-d+(d-1)\theta_0}
\}
$ taken over all $q\in \N\cap [1,(\log P)^{A}]$
and $a\in \Z\cap[0,q)$ with $\gcd(a,q)=1$,
then by Dirichlet's approximation theorem 
there
are coprime
integers $1\leq a'\leq q'$
with $q'\leq P^{d-(d-1)\theta_0}$
and $|\alpha-a'/q'|\leq P^{-d+(d-1)\theta_0}/q'$.
Thus
we must have 
$q'>(\log P)^{A}$. Alluding to Vaughan's estimate~\cite[\S 25]{Dav}
and using
partial summation
we obtain
\[
|W(\alpha)|\ll (P^dq'^{-1/2}+P^{4d/5}+(P^d q')^{1/2})(\log P)^3
\leq
 (P^d(\log P)^{-A/2}+P^{4d/5}+P^{d-\eta/2})(\log P)^3
,\]
which is $\ll P^d (\log P)^{-A/2+3}$.
For each $a$ and $q$ 
as
in~\eqref{tonny iommi > beethoven} 
 we get by~\eqref{eq:bound} that
\[
\int_{|\gamma|\leq P^{\eta}} 
|I(\c{B};\gamma)|
\frac{|W(a/q+\gamma P^{-d})|}{P^{d}}
\mathrm{d}\gamma
\ll
(\log P)^{-A/2+3}
,\] 
hence by
the second part of Proposition~\ref{deligne}
we see that
the sum over $q$ in the lemma is 
\[
\ll
\sum_{(\log P)^{A}<q\leq P^d}
\sum_{\substack{ a\in \Z\cap [0,q)\\ \gcd(a,q)=1}}
\frac{|S_{a,q}|}{q^n}
(\log P)^{-A/2+3}
\leq 
(\log P)^{-A/2+3+C'} 
.\qedhere\]
\end{proof}
\begin{lemma}
\label{lem:int565}
Assume~\eqref{halfassumption}.
Then we have
\[
\sum_{q\leq (\log P)^{A}}
q^{-n}
\sum_{\substack{ a\in \Z\cap [0,q)\\ \gcd(a,q)=1}}
|S_{a,q}|
\int_{(\log P)^{A}<|\gamma|\leq P^{\eta}} 
|I(\c{B};\gamma)|
\frac{|W(a/q+\gamma P^{-d})|}{P^{d}}
\mathrm{d}\gamma
\ll
\frac{(\log \log P)^{C'}}{
(\log P)^{\frac{A}{2^{d}(d-1)}}
}
.\]
\end{lemma}
\begin{proof}
The proof follows immediately by combining
the bound $W(\alpha)\ll P^d$, 
the inequality~\eqref{eq:bound} for $Q=(\log P)^{A}$
and the second part of Proposition~\ref{deligne}.
\end{proof}
Tying Lemmas~\ref{lem:vinogradov} and~\ref{lem:int565}
proves the following lemma.
\begin{lemma}
\label{lem:intermed}
Keep the assumptions
of
Theorem~\ref{thm:main}.
Then there exists a strictly
positive constant
$\lambda=\lambda(f)$
such that for every fixed
sufficiently large
$A>0$ we have 
\[
\Big|
\int_{\alpha \in \mathfrak{t}(A)}
S(\alpha)
\overline{W(\alpha)}
\da
\Big|
\ll
\frac{P^n}
{(\log P)^{A \lambda}}
.\]
\end{lemma}
\subsection{The major arcs}
\label{s:major}
Bringing together~\eqref{eq:ents}, \eqref{eq:entslalala},
and
Lemma~\ref{lem:intermed}
we see that under the assumptions of 
Theorem~\ref{thm:main}
there exists
$\lambda>0$ such that for all large $A>0$ we have 
\begin{equation}
\label{eq:entsbc}
\hspace{-0,1cm}
\frac{\pi_f(P\c{B})}{P^n}
-\hspace{-0,2cm}
\sum_{q\leq (\log P)^A}
\hspace{-0,2cm}
q^{-n}
\hspace{-0,2cm}
\sum_{\substack{ a\in \Z\cap [0,q)\\ \gcd(a,q)=1}}
\hspace{-0,1cm}
S_{a,q}
\int_{|\gamma|\leq (\log P)^A} 
\hspace{-0,1cm}
I(\c{B};\gamma)
\frac{\overline{W(a/q+\gamma P^{-d})}}{P^{d}}
\mathrm{d}\gamma
\ll
(\log P)^{-A\lambda }
.\end{equation}
Using 
the Siegel--Walfisz theorem as in~\cite[pg. 147]{Dav}
we can show 
that there exists $c=c(A)>0$
such that 
if
$|\beta|\leq P^{-d}(\log P)^{A}$,
$q\leq (\log P)^A$,
$a$ coprime to $q$
and 
$x\in [P^{d/2},P^{2d}]$
then   
\begin{equation} \label{eq:asdqwezxc}
\sum_{m\leq x} \Lambda(m)
\e(m(a/q+\beta))
=
\frac{\mu(q)}{\phi(q)}
\left(\int_{2}^x
\e(\beta t)
\mathrm{d}t
\right)
+O\left((1+|\beta| x) x\exp\left(-c \sqrt{\log P}\right)\right)
,\end{equation}
where $\mu,\phi$ and $\Lambda$ 
denote the 
M\"{o}bius, Euler and von Mangoldt functions.
We now see that 
\[
\sum_{p\leq x} (\log p)
\e(p(a/q+\beta))
=
\frac{\mu(q)}{\phi(q)}
\left(\int_{2}^x
\e(\beta t)
\mathrm{d}t
\right)
+O\left((1+|\beta| x) x\exp\left(-c \sqrt{\log P}\right)\right)
\]
due to 
the estimate 
$\sum_{m\leq x \atop m\neq p} \Lambda(m)\ll x^{1/2}$.
Partial summation 
shows that 
$W(a/q+\beta)$
equals
\[\frac{\mu(q)}{\phi(q)}
\!
\Bigg(
\!
\frac{\int_{2}^{2\max\{f_0(\c{B})\} P^d} \e(\beta t)\mathrm{d}t}{\log (\frac{1}{2}\max\{f_0(\c{B})\} P^d)}
-
\frac{\int_{2}^{\frac{1}{2}\min\{f_0(\c{B})\} P^d} \e(\beta t)\mathrm{d}t}{\log (\frac{1}{2}\min\{f_0(\c{B})\} P^d)}
-\int_{\frac{1}{2}\min\{f_0(\c{B})\} P^d}^{2\max\{f_0(\c{B})\} P^d}
\!
\!
\!
\Big(\int_{2}^u \e(\beta t)\mathrm{d}t\Big)
\!
\Big(\frac{1}{\log u}\Big)'
\!
\mathrm{d}u
\!
\Bigg) 
\]
up to an error of size
$
\ll
(1+|\beta| P^d) P^d\exp\left(-c \sqrt{\log P}\right)
$.
Partial integration now
yields 
   \begin{equation}    \label{vbdfhdfie}    
W(a/q+\gamma P^{-d})
=
\frac{\mu(q)}{\phi(q)}
\left(\int_{\frac{1}{2}\min\{f_0(\c{B})\} P^d}
^
{2\max\{f_0(\c{B})\} P^d}
\hspace{-0,2cm}
\frac{\e(\gamma P^{-d} t)\mathrm{d}t
}{\log t}
\right)
 +O\l(
 \frac{(1+|\gamma|)     P^{d} }{     {\exp\left(c \sqrt{\log P}\right)} }
 \r) .\end{equation}
 The error term makes the following contribution 
towards~\eqref{eq:entsbc},
\[
\ll
\exp\left(-c \sqrt{\log P}\right)
\sum_{q\leq (\log P)^A}
\hspace{-0,1cm}
q^{-n}
\hspace{-0,1cm}
\sum_{\substack{ a\in \Z\cap [0,q)\\ \gcd(a,q)=1}}
\hspace{-0,1cm}
|S_{a,q}|
\int_{|\gamma|\leq (\log P)^A} 
|I(\c{B};\gamma)| (1+(\log P)^A)
\mathrm{d}\gamma
\]
and, by the second part of 
Proposition~\ref{deligne}
this is 
$
\ll 
\exp\left(-c \sqrt{\log P}\right)
(\log P)^{A+1},$ which is obviously
$\ll
\exp\left(-c/2 \sqrt{\log P}\right)
$.
Hence, letting 
\[
\Xi_A(P):=
\sum_{q\leq (\log P)^A}
\frac{\mu(q)}{\phi(q)q^n}
\sum_{\substack{ a\in \Z\cap [0,q)\\ \gcd(a,q)=1}}
S_{a,q}
\] and \[ 
\Psi_A(P):=
\int_{|\gamma|\leq (\log P)^A} 
I(\c{B};\gamma)
\Bigg(\int_{\frac{1}{2}\min\{f_0(\c{B})\} P^d}
^
{2\max\{f_0(\c{B})\} P^d}
\frac{\e(-\gamma P^{-d} t)}{\log t}
\mathrm{d}t
\Bigg)
\mathrm{d}\gamma
,\]
we
obtain the following result via~\eqref{eq:entsbc}.
\begin{lemma}
\label{lem:finalnnn}
Under the assumptions of
Theorem~\ref{thm:main}
there exists $\lambda=\lambda(f)>0$ 
such that for every $A>0$
we have  
$
\pi_f(P\c{B})
=\Xi_A(P) 
\Psi_A(P)
P^{n-d}
+O(P^n
(\log P)^{-A\lambda})
$ for all sufficiently large $P.$
\end{lemma}

\subsection{The non-archimedean densities}
\label{s:nonarch}
If $n-\sigma_f\geq 3$ then
(\ref{Tfp}) along with the multiplicativity of $T_f$ \cite[\S 7]{birch} 
gives
\[\sum_{q>x}
\frac{|\mu(q)|}{\phi(q)q^n}
\sum_{
a\in (\Z/q\Z)^*
}
|S_{a,q}|
\leq 
\sum_{q>x}
\frac{|\mu(q)|}{\phi(q)}
q^{1-\frac{(n-\sigma_f)}{2}
+\epsilon}.
\]
Hence, for $q \in \N$,  the estimate
$
q/\phi(q)\ll    \log \log (4q)
$
that can be found for example in 
\cite[Th. 5.6]{MR3363366} implies
$$
\sum_{q>x}
\frac{|\mu(q)|}{\phi(q)q^n}
\sum_{
a\in (\Z/q\Z)^*
}
|S_{a,q}| \ll x^{-1/2+\varepsilon}.
$$
Therefore, we have
\[
\Xi_A(P)=
\sum_{q=1}^\infty
\frac{\mu(q)}{\phi(q)q^n}
\sum_{ 
a\in (\Z/q\Z)^* 
}
S_{a,q}
+O((\log P)^{-A/4})
.\]
The multiplicativity of the last sum over $a$
shows that the above
sum over $q$ is $\prod_p \beta_p$, where 
\[
\beta_p:=1-\frac{1}{(p-1)p^n}
\sum_{
a\in (\Z/p\Z)^*
}
S_{a,p}
.\]
Finally,
the 
following lemma is obtained by 
observing that 
\begin{equation}
\label{eq:observatory}
\sum_{
a\in (\Z/p\Z)^*
}
S_{a,p}
=
\sum_{\b{x}\in \F_p^n} 
\Big(-1+
\sum_{\substack{ a\in \Z/p\Z }} \e(af(\b{x})/p)
\Big)
=-p^n
+p\#\{\b{x}\in \F_p^n:f(\b{x})=0\} 
.\end{equation} 
\begin{lemma}
\label{lem:padicdenscr8}
If 
 $n-\sigma_f\geq 3$
then  
\[
\Xi_A(P)= \prod_p\Bigg(
\left(1-\frac{\#\{\b{x}\in \F_p^n:f(\b{x})=0\}}{p^n}\r)
\left(1-\frac{1}{p}\right)^{-1}
\Bigg)
+O((\log P)^{-A/4})
.\]
\end{lemma}
Combining~\eqref{Tfp}
and~\eqref{eq:observatory}
yields
$
p^{-n}\#\{
\b{x}\in \F_p^n:f(\b{x})=0\}
=
1/p
+O(p^{-(n-\sigma_f)/2})
$, thus verifying the following lemma.
\begin{lemma}
\label{lem:padicdens}
If $n-\sigma_f\geq 3$ then 
the product 
in Theorem~\ref{thm:main}
converges
absolutely.
\end{lemma} 

\subsection{The archimedean densities}
\label{s:arch}

Letting for $P^d>\frac{1}{2}\min\{f_0(\c{B})\} $
\[
\Psi(P):=\int_{\gamma \in \R} 
I(\c{B};\gamma)
\Bigg(\int_{\frac{1}{2}\min\{f_0(\c{B})\} }
^
{2\max\{f_0(\c{B})\} }
\frac{\e(-\gamma  \mu)}{\log (\mu P^d)}
\mathrm{d}\mu
\Bigg)
\mathrm{d}\gamma
,\]
we see by~\eqref{eq:bound}
and our
assumption~\eqref{halfassumption}
that 
there exists
$\lambda_2=\lambda_2(f)>0$
such that 
\begin{equation}
\label{eq:imedfurn}
\Psi_A(P)
P^{-d}=
\Psi(P)
+O_A((\log P)^{-\lambda_2 A})
.\end{equation}
Now we observe that for all
reals 
$z,\mu$ with 
 $z>\mu>0$ and 
$z\notin \{1/\mu,1\}$
we have 
\begin{equation}
\label{eq:taylors}
\frac{1}{\log (\mu z)}
=\frac{1}{\log z}
\frac{1}{\big(1+\frac{\log \mu}{\log z}\big)}
=\frac{1}{\log z}
\sum_{k=0}^\infty \frac{(-1)^k}{(\log z)^{k}}
(\log \mu)^k
,\end{equation}
therefore, letting 
 for $k\in \Z_{\geq 0}$, 
\begin{equation}
\label{eq:lethgj}
J(k):=\int_{\gamma \in \R} 
I(\c{B};\gamma)
\Bigg(\int_{\frac{1}{2}\min\{f_0(\c{B})\} }
^
{2\max\{f_0(\c{B})\} }
\e(-\gamma  \mu) 
(\log \mu)^k
\mathrm{d}\mu
\Bigg)
\mathrm{d}\gamma
,\end{equation}
we infer 
that for all sufficiently large $P$ 
we have 
\begin{equation}
\label{eq:psikoln}
\Psi(P)=
\frac{1}{\log (P^d)}
\sum_{k=0}^\infty \frac{(-1)^k}{(\log (P^d))^{k}}
J(k)
.\end{equation} 
Let us furthermore introduce 
the 
following 
entity for all $n\in \N$
and $k \in \Z_{\geq 0}$,
\begin{equation}
\label{eq:kouradi}
J_n(k)
:=\int_{\gamma \in \R} 
\exp\l(-\frac{\pi^2\gamma^2}{n^2}\r)
I(\c{B};\gamma)
\Bigg(\int_{\frac{1}{2}\min\{f_0(\c{B})\} }
^
{2\max\{f_0(\c{B})\} }
\e(-\gamma  \mu) 
(\log \mu)^k
\mathrm{d}\mu
\Bigg)
\mathrm{d}\gamma
.\end{equation} 
\begin{lemma}
\label
{lem:four}
Under the assumption~\eqref{halfassumption}
we have
$\displaystyle\lim_{n\to+\infty} J_n(k)=J(k)$
for every $k \in \Z_{\geq 0}$.
\end{lemma}
\begin{proof}  
We have $J(k)-J_n(k)\ll_k 
\int_{|\gamma|\leq \log n}    
\c H_\dagger
(\gamma) \mathrm{d}\gamma
+
\int_{|\gamma|> \log n}    
\c H_\dagger
(\gamma) \mathrm{d}\gamma
$, where 
\[
\c H_\dagger
(\gamma):=
\l(1-
\exp\l(-\frac{\pi^2\gamma^2}{n^2}\r)
\r)
|I(\c{B};\gamma)| 
.\] 
We have $I(\c{B};\gamma) \ll 1$ due to
Lemma~\ref{lem:birch5.2},
hence,
$$
\int_{|\gamma|\leq \log n}    
\c H_\dagger (\gamma) \mathrm{d}\gamma
\ll (\log n)
\l(1-
\exp\l(-\frac{\pi^2(\log n)^2}{n^2}\r)
\r)=o(1).$$
By~\eqref{eq:bound}
we get 
 $\int_{|\gamma| 
>
 \log n}    
\c H_\dagger
(\gamma) \mathrm{d}\gamma
\ll (\log n)^{-\lambda_1}=o(1)$
for some positive
$\lambda_1=\lambda_1(f)$.
\end{proof}

\begin{lemma}
\label{lem:prwinoskafes}
Under the assumption~\eqref{halfassumption}
we have the following for every
$k \in \Z_{\geq 0}$,
$$\lim_{n\to+\infty} J_n(k)=\int_{\b{t}\in \c{B}}
(\log f_0(\b{t}))^k
\mathrm{d}\b{t}.$$
\end{lemma}
\begin{proof}
It is standard to see that the Fourier transform of 
the function $\phi_n:\R\to \R$
defined through 
$\phi_n(x):=\pi^{-1/2} n \exp(-n^2 x^2)$ 
satisfies 
$\widehat{\phi}_n(\gamma)=\exp(-\pi^2 n^{-2} \gamma^2 )$.
Therefore, the Fourier inverse formula yields
$\phi_n(x)=\int_\R  \e(x\gamma) \widehat{\phi}_n(\gamma)
\mathrm{d}\gamma$.
Using this for $x=f_0(\b{t})-y$
and rewriting~\eqref{eq:kouradi}
as
\[
\int_{\b{t}\in \c{B}}
\int_{\frac{1}{2}\min\{f_0(\c{B})\} }^{2\max\{f_0(\c{B})\} }
(\log \mu)^k
\Bigg(
\int_{\gamma \in \R} 
\exp\l(-\frac{\pi^2\gamma^2}{n^2}\r)
\e((f(\b{t})-\mu)\gamma) 
\mathrm{d}\gamma
\Bigg)
\mathrm{d}\mu
\mathrm{d}\b{t}
,\]
we infer that 
$J_n(k)=\int_\c{B}
g_n(\b{t})
\mathrm{d}\b{t}
$,
where
\[
g_n(\b{t}):=
\int_{\frac{1}{2}\min\{f_0(\c{B})\} }^{2\max\{f_0(\c{B})\} }
(\log \mu)^k
\phi_n(f(\b{t})-\mu)
\mathrm{d}\mu
.\]
It is obvious from \cite[Ex. I.2]{MR0435832}
that for any reals
$a<c<b$ 
and any continuous function $h:[a,b]\to \R$
one has 
\[
\lim_{n\to+\infty}
\int_a^b
h(\mu)
\phi_n(c-\mu)
\mathrm{d}\mu
=h(c). 
\]
Recalling that 
$f_0(\c{B})\subset (0,\infty)$
we infer that 
whenever
$\b{t} \in \c{B}$
then 
the following
inequality 
holds,
$
\frac{1}{2}\min\{f_0(\c{B})\} 
<
f_0(\b{t})
<
2\max\{f(\c{B})\}
$.
This gives   
$  
\lim_{n\to+\infty} 
  g_n(\b{t})
=(\log f_0(\b{t}))^k
$
and a use of the dominated convergence theorem 
concludes the proof of the lemma.
\end{proof}
\begin{lemma}
\label{lem:biskotio} 
Under the assumption~\eqref{halfassumption}
we have, for all sufficiently large $P$,
$\Psi(P)
=P^{-n}
\mathrm{Li}_f(P\c{B})
$.
\end{lemma}
\begin{proof}
Combining Lemmas~\ref{lem:four}
and~\ref{lem:prwinoskafes}
we get 
$J(k)=\int_\c{B} (\log f_0(\b{t}))^k\mathrm{d}\b{t}$.
Injecting this into~\eqref{eq:psikoln}
and interchanging 
the sum over $k$ and the integral over $\b{t}$
yields 
\begin{equation} \label{eq:zombie ritual}
\Psi(P)=\int_\c{B} 
\Bigg(
\frac{1}{\log (P^d)}
\sum_{k=0}^\infty \frac{(-1)^k}{(\log (P^d))^{k}}
(\log f_0(\b{t}))^k
\Bigg)
\mathrm{d}\b{t}
.\end{equation}
The proof is concluded by
alluding to~\eqref{eq:taylors}
and making the
change of variables 
$\b{x}=P\b{t}$.
\end{proof}
Combining
Lemma~\ref{lem:biskotio}
with~\eqref{eq:imedfurn}
provides us with the following result.
\begin{lemma}
\label{lem:padicdensdu}
Under the assumptions of
Theorem~\ref{thm:main}
there exists $\lambda_2=\lambda_2(f)>0$ 
such that for every $A>0$ and every sufficiently large $P$
we have  
$$
\Psi_A(P)= \mathrm{Li}_f(P\c{B})
P^{-(n-d)}
+O_A(
P^{d}
(\log P)^{-A
\lambda_2}).
$$
\end{lemma}
Our final result offers an
asymptotic
expansion of $\mathrm{Li}_f(P\c{B})$
in terms of $(\log P)^{-1}$.
\begin{lemma}
\label{lem:laurent}
For $f$ and $\c{B}$ as in Theorem~\ref{thm:main} and $P$ large enough
we have
\[
\mathrm{Li}_f(P\c{B})
=
\frac{\mathrm{vol}(\c{B})}{d}
\frac{P^n}{
\log P
}
+P^n\sum_{k=2}^\infty \frac{(-1)^{k-1}}{d^k}
\l(
\int_{\c{B}}
(\log f_0(\b{t}))^{k-1} \mathrm{d}\b{t}\r)
\frac{1}{(\log P )^k}
.\]
In particular, we have 
\[\mathrm{Li}_f(P\c{B})
=
\frac{\mathrm{vol}(\c{B})}{d}
\frac{P^n}{\log P }
+O_{f,\c B}
\l(\frac{P^n}{(\log P)^2
}\r)
.\]
\end{lemma}
\begin{proof}

The first equality follows by combining Lemmas~\ref{lem:four} and \ref{lem:prwinoskafes} with~\eqref{eq:psikoln} and \eqref{eq:zombie ritual}.
To prove the second,
note that if $\log P>2$ then 
\[\sum_{k=2}^\infty \frac{(-1)^{k-1}}{d^k}
\l(
\int_{\c{B}}
(\log f_0(\b{t}))^{k-1} \mathrm{d}\b{t}\r)
\frac{1}{(\log P)^k}
\ll \sum_{k=2}^\infty \frac{1}{(\log P)^k}
<
\frac{1}{(\log P)^2}
\sum_{k=2}^\infty \frac{1}{2^{k-2}}
,\] thus concluding the proof.
\end{proof}

\subsection{The proof of Theorem~\ref{thm:main}}
\label{theproof}
It
follows
by
merging
Lemmas~\ref{lem:finalnnn}, \ref{lem:padicdenscr8}
and~\ref{lem:padicdensdu}.
\qed

\section{The proof of Theorem~\ref{thm:second}}
\label{s:abc56}
\subsection
{First steps and auxiliary estimates} 
Similarly as in~\S\ref{s:utilising}
we may write  
\[
\#\{\b{S}_f\cap P\c{B}\}
=
\int_0^1
S(\alpha)
\overline{Q(\alpha)}
\mathrm{d}\alpha
,\]
where  $S(\alpha)$ is defined 
in~\eqref{s}
and  
\[ 
Q(\alpha):=
\sum_{\substack{
m \text{ square-free}, m\neq 0
\\
 \min\{f_0(\c{B})\} -1 \leq m P^{-d} \leq   \max\{f_0(\c{B})\}  +1
}}
\e(\alpha m)
.\] 

We shall later need certain estimates concerning exponential sums taking values over square-free integers that we record here. For $\alpha \in \R$ and $N\in \R_{\geq 1}$ define 
\[
f_2(\alpha,N):=\sum_{1\leq n \leq N} \mu(n)^2 \e(\alpha n)
.\]
The following result is the very 
special case corresponding to the choices
$k=2$ and $p=3/2$
in the work of Keil~\cite{MR3043604}.
\begin{lemma}
[Keil~{\cite[Th. $1.2$]{MR3043604}}] 
\label{lem:eugke} 
We have
\[\int_0^1 
|f_2(\alpha,N)|^{3/2} \mathrm{d}\alpha
\ll N^{1/2}(\log N)^2
.\]
\end{lemma}
For  $p$  prime  
and $\ell,m$  non-negative integers
such that $m \leq \ell$,   define   $g(p^{\ell},p^m)$ by
\[p^{\ell}(1-p^{-2})g(p^{\ell},p^m)=\begin{cases} 
0, &\mbox{if } \ell \geq m\geq 2  , \\ 
1, &\mbox{if } m <\min\{2,\ell\}, \\  
1-p^{\ell-2},  & \mbox{if } \ell=m\leq 1.
\end{cases} 
\]
We extend this definition  by defining the following whenever $d,q\in \N$ are such that $d\mid q$,
\[g(q,d):=\prod_{p\mid q}
g(p^{\nu_p(q)},  
p^{\nu_p(d)} 
)
.\] 
We can now introduce the following entity for 
 $q\in \N$, 
\begin{equation}
\label{eq:defgg}
G(q):=\sum_{b=1}^q
\e(b/q)
g(q,\gcd(b,q))
.\end{equation}

Br\"udern, Granville, Perelli, Vaughan and Wooley
studied $Q(\alpha)$
in~\cite{MR1620828}.
\begin{lemma}
\label{lem:MR1620828} 
There exist absolute positive constants $\delta_1,\delta_2$
such that for all $q\in \N$ with $q\leq P^{\delta_1}$,
all $a\in \Z\cap [1,q)$, $d\in \N$,
$\gamma \in \R$
with $|\gamma| \leq P^{\delta_1}$
and all $c_1<c_2\in \R$ 
we have  
\[
\sum_{\substack{m \text{ square-free}, m\neq 0
\\ c_1\leq mP^{-d} \leq c_2}}
\e(m(a/q+\gamma P^{-d})
)= 
\frac{G(q)}{\zeta(2)}
\left(\int_{c_1 P^d}
^
{c_2 P^d}
\e(\gamma P^{-d} t) 
\mathrm{d}t
\right)
+O_{c_1,c_2}\left((1+|\gamma|) P^{d-\delta_2}\right)
,\]
where $\zeta$ denotes the Riemann
zeta function and the implied constant depends at most on $c_1$ and $c_2$.
\end{lemma}
\begin{proof}
We will show that there exists an absolute 
$\delta>0$ 
such that if
$|\beta|\leq P^{-d+\delta}$,
$q\leq P^{\delta }$,
$a$ is 
coprime to $q$
and 
$x\in [P^{d/2},P^{2d}]$
then  
\begin{equation}
\label
{eq:giveproof}
 \sum_{1\leq m\leq x} \mu(m)^2
\e(m(a/q+\beta))
=
\frac{G(q)}{\zeta(2)}
\left(\int_{1}^x
\e(\beta t)
\mathrm{d}t
\right)
+O\left((1+|\beta| x) x^{1-\delta }
\right)
,\end{equation}
 from which one can  deduce the asymptotic stated in the lemma 
in the same way as we deduced~\eqref{vbdfhdfie}
from~\eqref{eq:asdqwezxc}.
To prove~\eqref{eq:giveproof}
we first note that for all  $b$ and $q\in \N$ 
we have 
\begin{equation}
\label
{eq:defggg}
\sum_{\substack{ 
1\leq m \leq x
\\
m\equiv b \md{q}
\\
m \text{ square-free}
}}1
=\sum_{\substack{ 
1\leq m \leq x
\\
m\equiv b \md{q}
}}\sum_{d^2 \mid m } \mu(d)
=
\sum_{\substack{1\leq d \leq \sqrt{x}  \\ \gcd(q,d^2)\mid b } }\mu(d) 
\l(\frac{x\gcd(q,d^2)}{q d^2 }+O\l(1\r)\r)
\end{equation}
and completing the sum over $d$ 
gives  
$
\frac{x}{\zeta(2)} 
g(q,\gcd(b,q))
+
O(\sqrt{x})
$.
For $\gcd(a,q)=1$ we let  
\[ 
Z(x;q,a)
:=
\sum_{1\leq m\leq x} \mu(m)^2
\e(m a/q )
=\sum_{b=1}^q
\e( b a /q )  
\sum_{\substack{ 
1\leq m \leq x
\\
m \equiv b \md{q}  }}\mu(m)^2
\]
and use~\eqref{eq:defgg} and~\eqref{eq:defggg}
to get 
the following estimate 
with an absolute implied constant
for 
 $q\leq x$,
\begin{equation}\label
{eq:sigwalf2}
Z(x;q,a)
-
 \frac{G(q) x}{\zeta(2)} 
\ll
q\sqrt{x}
.\end{equation}
We
therefore obtain by partial summation that 
\[ 
\sum_{1\leq m\leq x} \mu(m)^2
\e(m(a/q+\beta))
=\e(x \beta)
Z(x;q,a)
-2 \pi i \beta 
\int_1^x  
\e(u \beta)
Z(u ;q,a)\mathrm{d} u 
\]
and by~\eqref{eq:sigwalf2}
this becomes 
\[
\frac{G(q)}{\zeta(2)}
\l(\int_1^x \e(\beta t )\mathrm{d}t\r)
+
 O\l( (1+|\beta| x)  q \sqrt{x}  \r )
,\] with an absolute implied constant. 
This proves~\eqref{eq:giveproof}
with $\delta=(2+d)^{-1}$.
Indeed, if $q\leq P^\delta$ then 
the equality 
$P^{\delta}=P^{\frac{d}{2} (\frac{1}{2}-\delta)}$ and the bound
$P^{\frac{d}{2}}\leq x$
yield
$q\leq x^{\frac{1}{2}-\delta}
$, 
i.e.  
$q\sqrt{x} \leq x^{1-\delta}$.
\end{proof}
Finally, the next result is shown in the proof of~\cite[Lem. 3.1]{MR1620828}. 
\begin{lemma}
[Br\"udern, Granville, Perelli, Vaughan and Wooley, {\cite[Lem. 3.1]{MR1620828}}] 
\label{lem:MR1620828b} The function $G$ is multiplicative, 
supported in cube-free integers and satisfies for all prime $p$ the identity
\[
G(p)=G(p^2)=-p^{-2}(1-p^{-2})^{-1}
.\] 
\end{lemma}

\subsection{Continuation of the proof}
Recalling the meaning of 
$\mathcal{M}(\theta)$ and 
$\mathcal{M}_{a,q}(\theta)$
in~(\ref{mtheta}) and (\ref{maq}),
we allude to
H\"{o}lder's inequality and Lemma~\ref{lem:eugke}
to obtain
\begin{align*}
\left|\int_{\alpha \notin \c{M}(\theta)} S(\alpha)\overline{Q(\alpha)}\da\right|
& \leq  \l(\int_{\alpha \notin \c{M}(\theta)} |S(\alpha)|^3\da\r)^{1/3}
\l(\int_0^1 |Q(\alpha)|^{3/2}\da\r)^{2/3} \\
& \leq 
\l(\int_{\alpha \notin \c{M}(\theta)} |S(\alpha)|^3\da\r)^{1/3}
P^{d/3}(\log P)^{4/3}
.\end{align*}
The proof of Lemma~\ref{lem:4.4b}
can be adapted straightforwardly to show that if
\begin{equation}
\label{eq:10erbar}
1>\delta+6 d \theta_0
\ \text{ and } \ 
\frac{n-\sigma_f}{2^{d-1}}-\frac{2}{3}(d-1)>\delta \theta_0^{-1}
\end{equation} 
then 
\[
\l(\int_{\alpha \notin \c{M}(\theta)} |S(\alpha)|^3\da\r)^{1/3}
\ll P^{n-\frac{d}{3}-\frac{\delta}{9}}
.\] Let $\eta:=(d-1)\theta_0$.
Under the assumptions of Theorem~\ref{thm:second}
and for
$\theta_0$ as in~\eqref{eq:10erbar},
one obtains the following inequality that is in analogy with Lemma~\ref{lem:4.4c}, 
\[
\#\{\b{S}_f\cap P\c{B}\}=
\sum_{q\leq P^{\eta}}
\sum_{\substack{ a\in \Z\cap [0,q)\\ \gcd(a,q)=1}}
\int_{\c{M}'_{a,q}(\theta_0)} 
S(\alpha)\overline{Q(\alpha)}
\da
+O\l(P^{n-\frac{\delta}{10}
}\r)
.\] 
Similarly as in the proof of~\eqref{eq:ents},
one may now acquire 
some $\delta_1=\delta_1(f)>0$
such that 
\begin{equation}\label{eq:thefanalg6783}
\frac{\#\{\b{S}_f\cap P\c{B}\}}{P^n}
-
\sum_{q\leq P^{\delta_1
}}
q^{-n}
\sum_{\substack{ a\in \Z\cap [0,q)\\ \gcd(a,q)=1}}
S_{a,q}
\int_{|\gamma|\leq P^{\eta}} 
I(\c{B};\gamma)
\frac{\overline{Q(a/q+\gamma P^{-d})}}{P^{d}}
\mathrm{d}\gamma
\ll
P^{-\delta_1}
.\end{equation} 
By Lemma~\ref{lem:MR1620828} 
we see that 
for suitably small 
$\delta_1$
and all 
$a$ as in~\eqref{eq:thefanalg6783}
and $q\leq P^{\delta_1}$
one has 
\[
Q(a/q+\gamma P^{-d})
=
\frac{G(q)}{\zeta(2)}
\left(\int_{(\min\{f_0(\c{B})\}-1) P^d}
^
{
(\max\{f_0(\c{B})\}+1) P^d}
\e(\gamma P^{-d} t) 
\mathrm{d}t
\right)
+O\left((1+|\gamma|) P^{d-\delta_2}\right)
.\] 
Therefore, as in the proof of Lemma~\ref{lem:finalnnn}, we may infer that there exists
a positive constant 
$\delta_3=\delta_3(f)$ such that the quantity 
$\#\{\b{S}_f\cap P\c{B}\}$
equals 
\begin{equation}\label{eq:cafg}    
\hspace{-0,2cm}
\frac{P^{n}}{\zeta(2)}
\! 
\Bigg(\sum_{q\leq P^{
\delta_1
}}
\frac{G(q)}{q^n}
\sum_{\substack{ a\in \Z\cap [0,q)\\ \gcd(a,q)=1}}
\hspace{-0,2cm}
S_{a,q}
\Bigg)
\!
\Bigg(
\int_{|\gamma|\leq P^{
\delta_1
}} 
\hspace{-0,2cm}
\frac{I(\c{B};\gamma)}{P^d}
\Bigg(\int_{(\min\{f_0(\c{B})\}-1) P^d}
^
{
(\max\{f_0(\c{B})\}+1) P^d}
\hspace{-0,1cm}
\e(-\gamma P^{-d} t) 
\mathrm{d}t
\Bigg) 
\mathrm{d}\gamma
\Bigg)
,\end{equation}
up to an error term which is 
$O (P^{n-\delta_3})$.
We shall now 
use Lemma~\ref{lem:MR1620828b}
to show that the sum over $q$ 
forms an absolutely convergent series.
Bringing into play
(\ref{Tfp})
and \cite[Lem. 25]{41}
we
obtain the bounds
\[
|G(p)T_f(p)| \ll p^{-1-(n-\sigma_f)/2}
\
\text{ and }
\
|G(p^2)
T_f(p^2)| \ll p^{-n+\sigma_f}
.\]
Hence, assuming
$n-\sigma_f \geq 2$, these two estimates  
allow to modify easily the proof of Proposition~\ref{deligne},
thereby showing that 
the 
abscissa
of convergence of the Dirichlet series of $|G(q)|T_f(q)$ is 
strictly negative. This provides  
$\delta_4=\delta_4(f)>0$
such that for all $x\geq 2$, one has 
$
\sum_{q>x}|G(q)|T_f(q) \ll x^{-\delta_4} 
$,
hence the sum over $q$ in~\eqref{eq:cafg} is
$\Pi'+O(P^{-\eta \delta_4})$, where $\Pi'$ is
\[
\sum_{q=1}^\infty
\frac{G(q)}{q^n}
\sum_{\substack{ a\in \Z\cap [0,q)\\ \gcd(a,q)=1}}
S_{a,q}
=\prod_p
\Bigg(
1
-p^{-2}(1-p^{-2})^{-1}
\Big(
\frac{1}{p^n}
\sum_{\substack{ a\in \Z\cap (0,p)}}
S_{a,p}
+
\frac{1}{p^{2n}}
\sum_{\substack{ a\in \Z\cap [0,p^2)\\ \gcd(a,p)=1}}
S_{a,p^2}
\Big)
\Bigg)
.\]
One can easily see, for example, by using 
orthogonality of characters of
$ \Z/p^2\Z$ to detect the condition $f(\mathbf{x})=0$,
 that 
\[\#\big\{\b{x}\in (\Z/p^2\Z)^n: f(\b{x})= 0 
 \big\}
=
p^{2(n-1)}
\Big(1+
\frac{1}{p^n}
\sum_{\substack{ a\in \Z\cap (0,p)}}
S_{a,p}
+
\frac{1}{p^{2n}}
\sum_{\substack{ a\in \Z\cap [0,p^2),p\nmid a }}
S_{a,p^2}
\Big)
,\]
from which we can show that $\Pi'/\zeta(2)$ is  
\[
\prod_p
\Bigg(
1-\frac{
\#\big\{\b{x}\in (\Z/p^2\Z)^n: f(\b{x})=0
 \big\}
}{p^{2n}}
\Bigg)
.\] 
This is in agreement with the infinite product in Theorem~\ref{thm:second}.

To deal with the integral in~\eqref{eq:cafg}
we observe that the transformation $t=P^d \mu$ gives 
\begin{equation}
\label
{eq:first bwv bound}
P^{-d}
\int_{(\min\{f_0(\c{B})\}-1) P^d}
^
{(\max\{f_0(\c{B})\}+1) P^d}
\e(-\gamma P^{-d} t) 
\mathrm{d}t
=
\int_{\min\{f_0(\c{B})\}-1}
^
{\max\{f_0(\c{B})\}+1}
\e(-\gamma \mu) 
\mathrm{d}\mu
\ll
\min\{1,|\gamma|^{-1}\}
,\end{equation}
hence
Lemma~\ref{lem:birch5.2}
shows that 
 the integral in~\eqref{eq:cafg} 
converges absolutely and 
equals
\begin{equation}
\label{eq:simplebb}
\int_{\gamma \in \R} 
I(\c{B};\gamma)
\Bigg(\int_{\min\{f_0(\c{B})\}-1 }
^
{\max\{f_0(\c{B})\}+1 }
\e(-\gamma  \mu)
\mathrm{d}\mu
\Bigg)
\mathrm{d}\gamma
+O(P^{- \delta_5})
\end{equation}
for some
$\delta_5=\delta_5(f)>0$. 

One can combine the bound~\eqref{eq:first bwv bound}
with Lemma~\ref{lem:birch5.2}
to show that the integral over $\gamma$ in~\eqref{eq:simplebb} equals $\mathrm{vol}(\c{B})$
using arguments that are entirely analogous with the case $k=0$ in Lemmas~\ref{lem:four} and~\ref{lem:prwinoskafes}. 
Thereby alluding to the well-known estimate
$$\#\{\Z^n\cap P\c{B}\}
=\mathrm{vol}(\c{B})P^n+O_{\c{B}}(P^{n-1})
 $$ allows us to
 conclude the proof of Theorem~\ref{thm:second}.

 \appendix

\section{The Bateman--Horn heuristics in many variables}
\label{s:bathorco} 
In this section 
we extend the 
Bateman--Horn 
heuristics
from the setting of univariate polynomials to that of 
polynomials with arbitrarily many variables; we do so 
because we were unable to 
find a reference for this extension in the  literature.

In $1958$, Schinzel \cite{SchS} formulated the
following
conjecture concerning prime values of univariate polynomials. 
\begin{conjecture}[\textbf{Schinzel's hypothesis H, \cite{SchS}}]
\label{con:sinz}
Let 
$f_1,\dots,f_r \in \mathbb{Z}[x]$ be univariate irreducible polynomials with
positive leading coefficient. If   \ 
$\prod_{i=1}^r
f_i$
has no repeated polynomial factors
and, for every prime $p$, there exists $x_p \in \mathbb{Z}$ such that 
$p \nmid f_1(x_p)\cdots f_r(x_p)$,
then there exist
infinitely many integers $m$ such that $f_1(m),\ldots,f_r(m)$ 
are all primes.
\end{conjecture}

This conjecture was later
refined by
Bateman and Horn~\cite{BatH} who, based on the 
Cram\'er
 model
and the heuristics behind the Hardy--Littlewood conjecture (see~\cite[pg. 6-8]{Gran}),
gave a quantitative version of Schinzel's conjecture.
\begin{conjecture}[\textbf{Bateman--Horn's conjecture, \cite{BatH}}]
\label{con:simp}
	Keep the assumptions of Conjecture~\ref{con:sinz}.
	Then the  number of integers  $m\in [1,P]$ such that 
	every $f_1(m),\ldots,f_r(m)$ is prime is asymptotically equivalent to the 
following quantity as $P\to+\infty$,
$$\left(
\prod_{p\ \text {prime }}
\frac{
(1-p^{-1}\#\{
x\in \F_p:f_1(x)\cdots f_r(x)=0\})
}
{(1-1/p)^r}
\right)
\frac{1}{\deg(f_1)\cdots \deg(f_r)}
\int_2^P\frac{\mathrm{d}x}{(\log x)^r}.$$ 
\end{conjecture}
The convergence of the infinite
product  
is established in \cite{BatH}
using the prime ideal theorem.
These two conjectures lie very deep 
and
imply a number of notoriously difficult conjectures as immediate corollaries (the twin primes conjecture
among others; see \cite{SchS} for a non exhaustive list of implications). 
There are
applications to the arithmetic of
algebraic varieties,
see~\cite{Colliot},~\cite{Arne} or \cite{WSko}, 
where
Schinzel's hypothesis is assumed in order to
prove that the Hasse principle and
weak approximation holds.

Recall that for a polynomial $f$ we denote by $f_0$
the top degree part of $f$.
Let us now record the 
multivariable
version of the Bateman--Horn conjecture.
\begin{conjecture}[\textbf{Extension of the Bateman--Horn conjecture}]
	\label{conj}
Assume that we are given irreducible
polynomials
$f_1,\dots,f_r \in \mathbb{Z}[x_1,\ldots,x_n]$ such that
$\prod_{i=1}^r f_i$
has no repeated polynomial factors. Moreover, we assume that
$\c{B}\subset \R^n$ is a non-empty box 
such that
$f_{i0}(\c{B})\subset (1,\infty)$ for all 
$i\in \{1,\ldots ,r\}$.
Denote 
by  $\pi_{f_1,\ldots,f_r}(P\c{B})$
the cardinality 
of the set of 
integer vectors 
$\b{x} \in \Z^n\cap P\c{B}$
for which every 
$f_1(\b{x}),\dots, f_r(\b{x}) $
is a positive
prime number.
Then 
 $\pi_{f_1,\ldots,f_r}(P\c{B})$
is asymptotic to the following quantity as $P\to+\infty$, 
$$ 
\left(
\prod_{p\ \text {prime }}
\frac{
(1-p^{-n}\#\{
\mathbf{x}\in \F_p^n:f_1(\mathbf{x}
)\cdots f_r(\mathbf{x})=0\})
}
{(1-1/p)^r}
\right) 
\int_{P\c{B}} \frac{\mathrm{d}\b{x}}{\prod_{i=1}^r \log f_{i0}(\b{x})} 
. $$  
\end{conjecture}
\begin{remark}
\label{rem:serre}
Before providing 
the heuristics behind Conjecture~\ref{conj}
let us note that one can prove that 
the product over $p$ converges.
Indeed, a version of the prime number theorem for schemes over $\Z$ that can be
found in the work of Serre~\cite[Cor. 7.13]{Serre} implies  that
\[
\sum_{p\leq x}
\#\{\mathbf{x}\in \F_p^n:f_1(\mathbf{x})\cdots f_r(\mathbf{x})=0\}
=
r\Bigg(\int_2^{x^{n}}\frac{\mathrm{d}t}{\log t}\Bigg)
+O\l (x^n \e^{-c\sqrt{\log x}}\r )
\] for some $c=c(f_1,\ldots,f_r)>0$. Now partial summation implies that
$$
\sum_{p\leq x}
p^{-n}\#\{\mathbf{x}\in \F_p^n:f_1(\mathbf{x})\cdots f_r(\mathbf{x})=0\}=
r\log\log x +C_1+O\left(\frac{1}{\log x}\right)
$$
for some constant $C_1$. Hence the following series
converges,
$$
\sum_{p}\left( p^{-n}\#\{\mathbf{x}\in \F_p^n:f_1(\mathbf{x})\cdots f_r(\mathbf{x})=0\}-\frac{r}{p}\right)
,$$
which in turn yields the convergence of the product over $p$ appearing in the statement of the conjecture.
\end{remark}
We end this section by 
adapting
the heuristics behind Conjecture~\ref{con:simp} to the multivariate case. 
Recall that the Cram\'er model asserts that a random positive integer $m$ of size $X$
has 
probability $1/\log X$ of being a prime. 
An analogous statement can be made if the extra condition 
that $m$ lies in a primitive arithmetic progression modulo
$q$ for some positive integer 
$q$
is added, in this case
the probability is 
$1/(\phi(q) \log X)$ owing to Dirichlet's theorem on primes in arithmetic progressions.
This implies
that 
for coprime
 $a,q$, the conditional
probability that a positive integer $m$ of size $X$ is prime provided that $m\equiv a \md{q}$
equals
\begin{equation}
\label{cramerica}
\mathrm{Prob}[m\sim X \text{ is a prime } |  \ m\equiv a\md{q}]
\approx
\frac{
1/(\phi(q)\log X)
}{
1/q}
=\frac{q}{\phi(q) \log X}
.\end{equation}
In the setting of Conjecture~\ref{con:simp}
observe that 
for typical
$\b{x}\in \Z^n$
the integer $f_i(\b{x})$
can be prime 
only if $f_i(\b{x})$ is coprime to all small primes.
Therefore, letting 
$z=z(P)$ be 
a function 
that slowly tends to infinity with $P$
and 
letting 
$\c{P}
:=\prod_{p\leq z}p$,
we see that 
\begin{equation}
\label{eq:acinto}
\frac{\pi_{f_1,\ldots,f_r}(P\c{B})}{\#\{\Z^n\cap P\c{B}\}}
\approx 
\sum_{\substack{\b{a} \in (\Z/\c{P}\Z)^n\\
\forall i \in \{1,\dots,r\},\hspace{1mm} f_i(\b{a})\in (\Z/\c{P}\Z)^{\times}
}}
\hspace{-0,3cm}
 \mathrm{Prob}[x_i\equiv a_i\md{\c{P}} \text{ for all  } 1\leq i \leq n]
\cdot
\P_{\b{a},\c{P}}
,\end{equation}
where $\P_{\b{a},\c{P}}$ denotes the joint
probability defined through
\[\P_{\b{a},\c{P}}:=
\mathrm{Prob}
[m_i\sim P^{\deg(f_i)} \text{ is a prime for all } 1\leq i \leq r \ |  \ m_i\equiv f_i(\b{a}) \md{\c{P}}]
.\]
This is because the integer 
$f_i(\b{x})$ is typically of size $P^{\deg(f_i)}$ when 
$\b{x} \in P\c{B}$ and the
values 
$f_i(\b{x})$ are thought to behave like a random integer $m_i$ lying
in the arithmetic progression $f_i(\b{a})\md{\c{P}}$, provided that $\b{x}\equiv \b{a}\md{\c{P}}$.
Note that for $i\neq j$ the polynomials $f_i$ and $f_j$ are coprime
due to the  assumption
that $\prod_i f_i$ has no repeated factors,
therefore it is reasonable to expect that for $i\neq j$ the integer values 
$f_i(\b{x})$ and $f_j(\b{x})$ behave independently.
This suggests that   
\[
\P_{\b{a},\c{P}}=
\prod_{i=1}^r
\mathrm{Prob}
[m_i\sim P^{\deg(f_i)
}\text{ is a prime } |  \ m_i\equiv f_i(\b{a})\md{\c{P}}]
\]
and by~\eqref{cramerica}
one now gets
$
\P_{\b{a},\c{P}}
=
\c{P}^r
\phi(\c{P})^{-r}(\log P)^{-r}
\prod_{i=1}^r (\deg(f_i))^{-1}
$.
Substituting  
this into~\eqref{eq:acinto}
and noting that 
$\mathrm{Prob}[x_i\equiv a_i\md{\c{P}}]=1/\c{P}$
yields
\[\frac{
\pi_{f_1,\ldots,f_r}(P\c{B})
}{\mathrm{vol}(\c{B})P^n}
\approx 
\Big(\frac{\c{P}}{\phi(\c{P})\log P}\Big)^r 
\frac{1}{\prod_{i=1}^r \deg(f_i)}
\frac{1}{\c{P}^n}
\hspace{-0,2cm}
\sum_{\substack{\b{a} \in (\Z/\c{P}\Z)^n\\
\forall i \in \{1,\dots,r\},\hspace{1mm} f_i(\b{a})\in (\Z/\c{P}\Z)^{\times} 
}}
1
.\]
 The sum over $\b{a}$
forms 
a multiplicative function of $\c{P}$
that can be evaluated as
\[\prod_{p\leq z}
\big(
p^n-\#\{\mathbf{x}\in \F_p^n:f_1(\b{x})\cdots f_r(\b{x})=0\}
\big)
.\]
Putting everything together shows that 
we expect
$\pi_{f_1,\ldots,f_r}(P\c{B})$
to be approximated by
\[
\frac{\mathrm{vol}(\c{B})P^n}{(\log P)^r\prod_{i=1}^r \deg(f_i)}
\prod_{p\leq z}
\Bigg(
\bigg(\frac{p}{p-1}\bigg)^r
\bigg(\frac{p^n-\#\{\mathbf{x}\in \F_p^n:f_1(\b{x})\cdots f_r(\b{x})=0\}}{p^n}\bigg)
\Bigg)
.\]
In view of Remark~\ref{rem:serre} the product over $p\leq z(P)$ converges 
to the product in Conjecture~\ref{conj}
as $P\to+\infty$.
For $\b{x}\in   P\c{B}$
we have 
$f_{i0}(\b{x}) \asymp
 P^{\deg(f_{i0})}$
and using
$\deg(f_i)=\deg(f_{i0})$
we get
\[
\frac{\mathrm{vol}(\c{B})P^n}{(\log P)^r\prod_{i=1}^r \deg(f_i)}
=
\frac{\int_{P\c{B}} 1 \mathrm{d}\b{x}}{\prod_{i=1}^r \log (P^{\deg(f_i)})}
\asymp
\int_{P\c{B}} \frac{\mathrm{d}\b{x}}{\prod_{i=1}^r \log f_{i0}(\b{x})} 
,\]
thereby
concluding our explanation of the 
asymptotic in Conjecture~\ref{conj}.

\end{document}